\documentclass[onside,11pt,a4paper,reqno]{amsart}
\usepackage[left=1.4in, right=1.4in, top=1.5in, bottom=1.5in]{geometry}
\newtheorem{theorem}{Theorem}[section] 
\newtheorem{lemma}[theorem]{Lemma}
\newtheorem{proposition}[theorem]{Proposition} 
\newtheorem{remark}[theorem]{Remark} 
\usepackage{amsmath,amssymb,mathrsfs,indentfirst,cases,enumerate,bbm,upgreek,bm} 
\usepackage{xcolor,ulem}  
\usepackage[linkcolor=blue,citecolor=cyan,colorlinks]{hyperref}
\numberwithin{equation}{section}
\allowdisplaybreaks
\renewcommand{\ge}{\geqslant}
\renewcommand{\le}{\leqslant}
\newcommand{\identity}{\mathbf{I}} 
\newcommand{\tangentdiv}{\overline{\nabla} \cdot} 
\newcommand{\Laplace}{\Delta} 
\newcommand{\Divergence}{\operatorname{div}} 
\newcommand{\parl}{\partial} 
\newcommand{\norm}[1]{\left\lVert #1 \right\rVert} 
\newcommand{\module}[1]{\left| #1 \right|} 
\newcommand{\Brace}[1]{\left\{ #1 \right\}} 
\newcommand{\Paren}[1]{\left( #1 \right)} 
\newcommand{\Bracket}[1]{\left[ #1 \right]} 
\newcommand{\Anglebracket}[1]{\left< #1 \right>} 
\newcommand{\Zero}{\mathbf{0}} 
\newcommand{\vare}{\varepsilon}
\newcommand{\bPsi}{\boldsymbol{\Psi}}
\newcommand{\bPi}{\boldsymbol{\Pi}}
\newcommand{\bPhi}{\boldsymbol{\Phi}}
\newcommand{\Up}{{\Upsilon}}  
\newcommand{\eigv}{\zeta} 
\newcommand{\coeff}{\xi} 
\newcommand{\SSS}{\mathbf{S}} 
\newcommand{\DDD}{\mathbf{D}}
\newcommand{\VV}{\mathbf{v}}

\newcommand{\zz}{\mathbf{z}}
\newcommand{\ww}{\mathbf{w}}
\newcommand{\UU}{\mathbf{U}}
\newcommand{\xx}{\mathbf{x}}
\newcommand{\yy}{\mathbf{y}} 
\newcommand{\normal}{\boldsymbol{\nu}} 
\newcommand{\bR}{\overline{R}} 
\newcommand{\barsigma}{\bar{\sigma}} 
\newcommand{\surftenratio}{\sigma_{\operatorname{ratio}}} 
\newcommand{\bV}{\overline{V}} 
\newcommand{\Omp}{\Omega_{l}} 
\newcommand{\Tconst}{T_c} 
\newcommand{\Omn}{\Omega_{g}} 
\newcommand{\ldens}{\rho_l} 
\newcommand{\gdens}{\rho} 
\newcommand{\brho}{\bar{\gdens}} 
\newcommand{\GASMASS}{\operatorname{M}} 
\newcommand{\GASCONST}{\mathfrak{R}} 
\newcommand{\Heatcapa}{\mathfrak{c}} 
\newcommand{\Thermal}{\kappa} 
\newcommand{\Adi}{\gamma} 
\newcommand{\meancurv}{H} 
\newcommand{\Errorrho}{\varrho} 
\newcommand{\ErrorR}{\mathcal{R}} 
\newcommand{\ptErrorR}{\dot{\ErrorR}} 
\newcommand{\ptR}{\dot{R}} 
\newcommand{\pttR}{\ddot{R}} 
\newcommand{\ptbR}{\dot{\bR}} 
\newcommand{\Energy}{E} 
\newcommand{\Linearpart}{\mathcal{L}} 
\newcommand{\Nonlpart}{\mathcal{N}} 
\newcommand{\bkappa}{\chi} 
\newcommand{\ssd}{\gdens_{\dagger}} 
\newcommand{\ssr}{R_{\dagger}} 
\newcommand{\ssor}{\overline{R}_{\dagger}} 
\newcommand{\tR}{\tilde{R}_{\dagger}} 
\newcommand{\hatd}{\hat{\gdens}} 
\newcommand{\hatr}{\hat{R}} 
\begin{document}
\title[Nonlinear and exponential stability of a gas-liquid problem]{Exponential stability of a free boundary problem with spherical symmetry for a gas bubble immersed in a bounded incompressible liquid}
\author{Chengchun Hao, Tao Luo, \and Siqi Yang}
\address{Chengchun Hao and Siqi Yang\newline
	HLM, Institute of Mathematics, Academy of Mathematics and Systems Science, Chinese Academy of Sciences, Beijing 100190, China\newline
	\and
	School of Mathematical Sciences, University of Chinese Academy of Sciences, Beijing 100049, China
	}
\address{Tao Luo\newline
	Department of Mathematics, City University of Hong Kong, Hong Kong, China}

\email{hcc@amss.ac.cn}
\email{taoluo@cityu.edu.hk}
\email{yangsiqi@amss.ac.cn}
\begin{abstract}
This paper is mainly concerned with the free boundary problem for an approximate model (for example, arising from the study of sonoluminescence) of a gas bubble of finite mass enclosed within a bounded incompressible viscous liquid, accounting for surface tensions at both the gas-liquid interface and the external free surface of the entire gas-liquid region. It is found that any regular spherically symmetric steady-state solution is characterized by a positive root of a ninth-degree polynomial for which the existence and uniqueness are proved and a one-to-one correspondence between equilibria and pairs of gas mass and liquid volume is established. We prove that these equilibria exhibit nonlinear and exponential asymptotic stability under small perturbations that conserve gas mass and liquid volume, and an equilibrium solution acts as a local minimizer of the energy functional, even under relatively large perturbations, with the proportionality constant determined by the adiabatic constant. Moreover, we construct a global center manifold to apply the center manifold theory. Our results apply to gases and liquids of all sizes. Furthermore, we derive the optimal exponential decay rate for small liquid volumes by analyzing the spectrum bounds of the associated linear operator and show that decreasing the gas mass or increasing the temperature can accelerate the convergence rate, a behavior not seen in unbounded liquid scenarios.
\end{abstract}
\maketitle
\tableofcontents

\section{Introduction}   

Consider the free boundary problem of a gas bubble immersed in an incompressible viscous liquid with finite gas mass and liquid volume. The external liquid dynamics are governed by the following incompressible Navier-Stokes equations: 
\begin{subnumcases}
{\label{eq-l-full}} 
\parl_t(\ldens\VV_l)+\Divergence(\ldens\VV_l\otimes\VV_l)=\Divergence(-p_l\identity+\SSS_l(\VV_l)),&\text{in $\Omp(t)$,}\label{eq-l-full-1}\\
\Divergence\VV_l=0,&\text{in $\Omp(t)$,}\label{eq-l-full-2}\\
\ldens c_l\Bracket{\parl_tT_l+\Paren{\VV_l\cdot\nabla}T_l}=\Divergence(\kappa_l\nabla T_l)+\SSS_l(\VV_l):\nabla\VV_l,&\text{in $\Omp(t)$,}\label{eq-l-full-3}
\end{subnumcases} 
where $\Omp(t)\subset\mathbb{R}^3$ is a bounded connected domain that varies with time. The constant $\ldens>0$ represents the density of the liquid. The variables $\VV_l,p_l$, and $T_l$ denote the velocity field, pressure, and temperature of the liquid, respectively. The tensor product is denoted by $\otimes$, and $\identity$ represents the identity matrix. The viscous stress tensor is given by $\SSS_l(\VV_l)=2\mu_l \DDD(\VV_l)$, where $\mu_l>0$ is the dynamic viscosity and $\DDD(\VV)$ is the symmetric part of the gradient of the velocity field, defined as $(\nabla\VV+(\nabla\VV)^\top)/2$. The constants $c_l$ and $\kappa_l$ denote the specific heat capacity and thermal conductivity of the liquid, respectively. Additionally, $\VV_l\cdot\nabla$ represents the directional derivative, and $\mathbf{A}:\mathbf{B}$ denotes the trace of the matrix product $\mathbf{A}\mathbf{B}^\top$.
In the following, vectors, matrices, and tensors will be represented using bold typefaces. 

The internal gas is governed by the following compressible Navier-Stokes equations in a simply connected domain (see, e.g., \cite{Lai2023}):
\begin{subnumcases}{} 
\parl_t\gdens+\Divergence(\gdens\VV_g)=0,&\text{in $\Omn(t)$,}\label{eq-g-full-1}\\
\parl_t(\gdens\VV_g)+\Divergence(\gdens\VV_g\otimes\VV_g)=\Divergence\Paren{-p_g\identity+\SSS_g(\VV_g)},&\text{in $\Omn(t)$,}\label{eq-g-full-2}\\
\gdens T_g\Paren{\parl_ts+\VV_g\cdot\nabla s}=\Divergence(\Thermal\nabla T_g)+\SSS_g(\VV_g):\nabla\VV_g,&\text{in $\Omn(t)$,}\label{eq-g-full-3}\\
p_g=\GASCONST T_g\gdens,&\text{in $\Omn(t)$,}\label{eq-g-full-4}\\
s=\Heatcapa\log\Paren{{p_g}/{\gdens^\Adi}},&\text{in $\Omn(t)$,}\label{eq-g-full-5} 
\end{subnumcases}
where the variables $\gdens,\VV_g,p_g,T_g$, and $s$ denote the density, velocity field, pressure, temperature, and entropy per unit of mass of the gas, respectively. The viscous tensor 
\begin{equation*}
\SSS_g(\VV_g)=2\mu\Paren{\DDD(\VV_g)-\frac{1}{3}(\Divergence\VV_g)\identity}+\zeta(\Divergence\VV_g)\identity,
\end{equation*} 
where $\mu$ denotes the dynamic viscosity, and $\zeta$ represents the bulk viscosity.  The positive constants $\Thermal,\Heatcapa$ and $\GASCONST$ denote the thermal conductivity, specific heat capacity, and the ratio of the ideal gas constant to the molar mass, respectively. The adiabatic constant $\Adi=1+\GASCONST/\Heatcapa$, which is $5/3$ for monoatomic gases and $7/5$ for diatomic gases (see, e.g., \cite{Biro2000}), is also included. Moreover, \eqref{eq-g-full-4} and \eqref{eq-g-full-5} follow from Boyle's law, Joule's second law, and the second law of thermodynamics for ideal gases. 

We take into account the surface tension acting on the gas-liquid interface $\parl\Omn(t)$ as well as on the free boundary of the entire gas-liquid region, specifically $\parl\Omp(t)\setminus\parl\Omn(t)$. The boundary conditions on the gas-liquid interface read
\begin{subnumcases}
{\label{eq-lg-full}}  
\VV_l\cdot\normal=\VV_g\cdot\normal=v_{-},&\text{on $\parl\Omn(t)$,}\label{eq-lg-full-1}\\
\normal\cdot\Bracket{-p_l\identity+\SSS_l(\VV_l)-\Paren{-p_g\identity+\SSS_g(\VV_g)}}=-\sigma\meancurv\normal,&\text{on $\parl\Omn(t)$,}\label{eq-lg-full-2}\\
T_g=T_l,&\text{on $\parl\Omn(t)$,}\label{eq-lg-full-3} 
\end{subnumcases}
where $\normal$ denotes the unit outer normal to $\parl\Omn(t)$, and $v_{-}$ is the normal velocity of the interface. The mean curvature is given by $\meancurv=-\tangentdiv\normal$, where $\overline{\nabla}$ represents the tangential gradient (see, e.g., \cite{Koehne2013}), and the constant $\sigma>0$ is the surface tension coefficient. Similarly, on the external free surface, it holds
\begin{subnumcases}{\label{eq-lo-full}}  
\VV_l\cdot\normal=v_{+},&\text{on $\parl\Omp(t)\setminus\parl\Omn(t)$,}\label{eq-lo-full-1}\\
\normal\cdot\Paren{-p_l\identity+\SSS_l(\VV_l)}=\barsigma\meancurv\normal,&\text{on $\parl\Omp(t)\setminus\parl\Omn(t)$,}\label{eq-lo-full-2}\\
T_l=\Tconst,&\text{on $\parl\Omp(t)\setminus\parl\Omn(t)$,}\label{eq-lo-full-3}
\end{subnumcases} 
where $\normal$ denotes the unit outer normal, and $v_{+}$ represents the normal velocity of the free surface. The constant $\barsigma>0$ denotes the surface tension coefficient of the external free boundary $\parl\Omp(t)\setminus\parl\Omn(t)$, with $\sigma\neq\barsigma$ in general. Moreover, the temperature outside the gas-liquid region remains constant at the temperature $\Tconst$, which is positive and provides the continuity condition \eqref{eq-lo-full-3}. 
Finally, we assume that the compatibility conditions hold for the initial data $\VV_l(\cdot,0), T_l(\cdot,0),\gdens(\cdot,0),\VV_g(\cdot,0)$ and $p_g(\cdot,0)$ in the initial gas-liquid domain $\overline{\Omn(0)\cup\Omp(0)}$.

Since we are considering the scenario where the liquid volume is finite, the divergence-free condition \eqref{eq-l-full-2} ensures that the liquid volume remains constant
\begin{equation*}
\module{\Omp(t)}\equiv V\in (0,\infty),\quad t\ge 0.
\end{equation*} 
Moreover, by the conservation of mass \eqref{eq-g-full-1}, we have 
\begin{equation}\label{eq-conserve}
\int_{\Omn(t)}\gdens(x,t)dx=\int_{\Omn(0)}\gdens_0(x)dx,\quad t>0,
\end{equation} 
where $\gdens_0(\cdot)=\gdens(\cdot,0)$ is the initial gas density.
\subsection*{The  approximate model of the full free boundary problem \eqref{eq-l-full}--\eqref{eq-lo-full}}

This paper is primarily concerned with the stability of equilibrium (time-independent) solutions for the following approximate model \eqref{eq-l-main}--\eqref{eq-b-main}, which arises, for instance, in the study of sonoluminescence \cite{Barber1992,Barber1994}. 
This model reads
\begin{subnumcases}{\label{eq-l-main}}  
\ldens\Paren{\parl_t\VV_l+\VV_l\cdot\nabla\VV_l}-\mu_l\Laplace\VV_l +\nabla p_l=0,&\text{in $\Omp(t)$,}\label{eq-l-main-1}\\
\Divergence\VV_l=0,&\text{in $\Omp(t)$,}\label{eq-l-main-2}\\
T_l\equiv\Tconst,&\text{in $\Omp(t)$,}\label{eq-l-main-3} 
\end{subnumcases} 
where $\Tconst$ is a prescribed temperature consistent with the boundary condition \eqref{eq-lo-full-3},
\begin{subnumcases}{\label{eq-g-main}}  
\parl_t\gdens+\Divergence(\gdens\VV_g)=0,&\text{in $\Omn(t)$,}\label{eq-g-main-1}\\
p_g=p_g(t),&\text{in $\Omn(t)$,}\label{eq-g-main-2}\\
\gdens T_g\Paren{\parl_ts+\VV_g\cdot\nabla s}=\Divergence(\Thermal\nabla T_g),&\text{in $\Omn(t)$,}\label{eq-g-main-3}\\
p_g=\GASCONST T_g\gdens,&\text{in $\Omn(t)$,} \label{eq-g-main-4}\\
s=\Heatcapa\log\Paren{{p_g}/{\gdens^\Adi}},&\text{in $\Omn(t)$,}\label{eq-g-main-5}  
\end{subnumcases} 
and the boundary conditions
\begin{subnumcases}{\label{eq-b-main}}  
\VV_l\cdot\normal=\VV_g\cdot\normal=v_{-},&\text{on $\parl\Omn(t)$,}\label{eq-b-main-1}\\
\Paren{p_g-p_l}\normal+2\mu_l\normal\cdot\DDD(\VV_l)=-\sigma\meancurv\normal,&\text{on $\parl\Omn(t)$,}\label{eq-b-main-2}\\
T_g=T_l,&\text{on $\parl\Omn(t)$,}\label{eq-b-main-3}\\
\VV_l\cdot\normal=v_{+},&\text{on $\parl\Omp(t)\setminus\parl\Omn(t)$,}\label{eq-b-main-4}\\
-p_l\normal+2\mu_l\normal\cdot\DDD(\VV_l)=\barsigma\meancurv\normal,&\text{on $\parl\Omp(t)\setminus\parl\Omn(t)$.}\label{eq-b-main-5} 
\end{subnumcases}   

The above approximation gas system \eqref{eq-g-main} can be found in \cite{Biro2000,Lai2023}. As noted in \cite[Section 2]{Biro2000}, the gas pressure equation \eqref{eq-g-main-2} is derived under the high sound velocity assumption. This implies that disturbances in the fluid propagate quickly through the gas region, allowing the momentum equation \eqref{eq-g-full-2}  to be simplified to $\nabla p_g\equiv\Zero$ in $\Omn(t)$.  

For gas dynamic \eqref{eq-g-main}, if we substitute \eqref{eq-g-main-5} into \eqref{eq-g-main-3} and eliminate the temperature $T_g$ by using \eqref{eq-g-main-4}, it holds
\begin{equation*}
\Paren{\parl_t+\VV_g\cdot\nabla}\Bracket{\Heatcapa\Paren{\log{p_g}-\log{\gdens^\Adi}}}=\Thermal\Laplace\Paren{\gdens^{-1}},\quad \text{in}\ \Omn(t).
\end{equation*}
Upon simplification by applying \eqref{eq-g-main-1}, we obtain 
\begin{equation}\label{eq-gas-velocity}
\dot p_g/p_g=\Thermal\Heatcapa^{-1}\Laplace\Paren{\gdens^{-1}}-\Adi\Divergence\VV_g,\quad \text{in}\ \Omn(t),
\end{equation}
where we have denoted $\dot p_g=\parl_t p_g$. This implies that the divergence of the gas velocity is determined by both the gas density and the pressure at the gas-liquid interface. Therefore, system \eqref{eq-g-main} is equivalent to
\begin{subnumcases}{\label{eq-g-main-equivalent}}  
\parl_t\gdens+\Divergence(\gdens\VV_g)=0,&\text{in $\Omn(t)$,}\label{eq-g-main-equivalent-1}\\
p_g=p_g(t),&\text{in $\Omn(t)$,}\label{eq-g-main-equivalent-2}\\
\dot p_g/p_g=\Thermal\Heatcapa^{-1}\Laplace\Paren{\gdens^{-1}}-\Adi\Divergence\VV_g,&\text{in $\Omn(t)$.}\label{eq-g-main-equivalent-3} 
\end{subnumcases}
Furthermore, by eliminating the divergence of the gas velocity, one has
\begin{equation}\label{eq-parl-gas density}
\parl_t\gdens=\frac{\Thermal}{\Adi\Heatcapa}\Paren{\Laplace\log\gdens-\frac{\module{\nabla\gdens}^2}{\gdens^2}}-\VV_g\cdot\nabla\gdens+\frac{\dot p_g}{\Adi p_g}\gdens,\quad \text{in}\ \Omn(t).
\end{equation}

From the form of the equivalent gas system \eqref{eq-g-main-equivalent}, it is evident that without imposing the irrotational condition ($\operatorname{curl}\VV_g=0$) or any symmetry assumptions, the approximate model \eqref{eq-l-main}--\eqref{eq-b-main}, in general, may not have a unique solution. Even for steady-state solutions, except for the spherically symmetric equilibria, other solutions involve rigid rotations. Additionally, when considering steady-state solutions, the liquid viscosity and the presence of surface tension require that both the gas bubble and the gas-liquid region exhibit spherical symmetry, as per Alexandrov's theorem. A recent study \cite{Lai2024a} shows that assuming the flows are irrotational, the shape of any steady-state gas is exclusively spherical by the surface tension alone.

Therefore, our focus will be on spherically symmetric solutions to the free boundary problem \eqref{eq-l-main}--\eqref{eq-b-main}. For a spherically symmetric solution, the gas domain $\Omn(t)$ is assumed to be a ball centered at the origin with radius $R(t)$. Given the density $\gdens$ of the gas in the spherical region $B_R$, we define the corresponding mass as 
\begin{equation*} 
\GASMASS[\gdens,R]=\int_{B_R}\gdens(x)dx.
\end{equation*}
\subsection*{Main results} 

We will study the stability of the spherically symmetric equilibria of free boundary problem \eqref{eq-l-main}--\eqref{eq-b-main} under the assumption of spherical symmetry. Under this assumption, system \eqref{eq-l-main}--\eqref{eq-b-main} is equivalent to problem \eqref{eq-main-syst} for the density and radius of the gas bubble, as will be shown in Proposition \ref{p-main-syst}. Based on this, we establish a one-to-one correspondence between the equilibrium gas state $(\gdens,R)\equiv(\ssd,\ssr)$ and the mass-volume pair $(M,V)$ for system \eqref{eq-l-main}--\eqref{eq-b-main}.
\begin{theorem}\label{t-equilibria} 
There exists a smooth bijective mapping $(M,V)\in(0,\infty)\times(0,\infty)\mapsto (\ssr[M,V],\ssd[M,V])$, as defined in \eqref{eq-Rstar-bound} and \eqref{eq-rhostar}, such that any regular spherically symmetric equilibrium solution to system \eqref{eq-l-main}--\eqref{eq-b-main} is uniquely determined by the mass-volume pair $(M,V)$. 
\end{theorem}  
\begin{remark}
If no confusion arises, we use subscript $\dagger$ to indicate the steady-state solutions or other related quantities. For instance, $\ssr$ represents the equilibrium gas radius and $\Energy_\dagger$ denotes the equilibrium energy.
\end{remark}
\begin{remark}
Based on Theorem \ref{t-equilibria}, we will show in Appendix \ref{appendix-1} that,  for the original full free boundary problem \eqref{eq-l-full}--\eqref{eq-lo-full}, the mass-volume pairs are the sole determinants of the regular spherically symmetric equilibria provided that the liquid temperature remains constant.
\end{remark}

To state our stability result, we introduce the following manifold of equilibria to system \eqref{eq-main-syst}, which is parameterized by the gas mass and liquid volume
\begin{equation}\label{eq-Mstar}
\Sigma=\Brace{\Paren{\gdens(x,t),R(t)}\equiv\Paren{\ssd[M,V],\ssr[M,V]}\mid 0<M,V<\infty},
\end{equation} 
where $\ssd$ and $\ssr$ are smooth functions defined in Theorem \ref{t-equilibria}. 

The main result of this paper is to demonstrate the nonlinear and exponential asymptotic stability of spherically symmetric equilibria $\Sigma$ to small perturbations. This is significant because the spherically symmetric solution to system \eqref{eq-l-main}--\eqref{eq-b-main} can be reconstructed from the density and radius of the gas bubble. Therefore, we can conclude that the equilibria of the free boundary problem \eqref{eq-l-main}--\eqref{eq-b-main} are also nonlinearly and exponentially asymptotically stable. It should be noted that in the following theorem, $B_1$ refers to the unit open ball and the norm $\norm{\cdot}_{C^{2+2\alpha}_y(B_1)}$ is defined in Appendix \ref{appendix-2}, where the global existence and the uniqueness of spherically symmetric solutions are also provided.

\begin{theorem}\label{t-main result}
Given any mass of the gas, $M>0$, and any liquid volume $V>0$. For free boundary problem \eqref{eq-main-syst}  with the liquid volume $V$, there exists a constant $\eta_0>0$ such that the following holds: 
\begin{enumerate}[(i)]
\item For any initial data $(\gdens_0,R_0,\ptR_0)$ such that the mass $\GASMASS[\gdens_0,R_0]=M$ and 
\begin{equation}\label{eq-main result-condition}
\norm{\gdens_0(R_0y)-\ssd[M,V]}_{C^{2+2\alpha}_y(B_1)}+\module{R_0-\ssr[M,V]}+|\ptR_0|\le\eta_0,
\end{equation}
where $\alpha\in(0,1/2)$ and the equilibrium $(\ssd[M,V],\ssr[M,V])\in \Sigma$, the global-in-time solution $(\gdens(r,t),R(t))$ satisfies	\begin{align}\label{eq-main- results-i}
&\norm{\gdens(R(t)y,t)-\ssd[M,V]}_{C^{2+2\alpha}_y(B_1)}\nonumber\\
&\quad+\module{R(t)-\ssr[M,V]}+|\ptR|+|\pttR|+|\dddot R|\rightarrow 0,\quad \text{as}\ t\rightarrow\infty.
\end{align} 
\item The global solution $(\gdens(r,t),R(t))$ converges to the equilibrium at an exponential rate. More precisely, there exists a constant $\varpi_\dagger>0$, such that 
\begin{align*}
&\norm{\gdens(R(t)y,t)-\ssd[M,V]}_{C^{2+2\alpha}_y(B_1)}\\
&\quad +\module{R(t)-\ssr[M,V]}+|\ptR|=O\Paren{e^{-\varpi_\dagger t}},\quad \text{as}\ t\rightarrow\infty.
\end{align*}
\end{enumerate}

Therefore, any spherical equilibrium solution of free boundary problem \eqref{eq-l-main}--\eqref{eq-b-main} is nonlinearly and exponentially asymptotically stable.  
\end{theorem}   
\subsection*{Background and history}

The dynamics of gas bubbles immersed in a liquid with finite volume is a multifaceted subject that intersects fluid dynamics, thermodynamics, and material science. This topic has significant implications in various fields such as industrial engineering \cite{Hashmi2012}, environmental science \cite{Leifer2003}, and biomedical applications  \cite{Chowdhury2020,Xiong2021}. Other discussion of bubble phenomena and applications can be found in the review article \cite{Prosperetti2004}. 

Gas bubbles in liquids exhibit complex behavior due to the interplay between pressure, surface tension, and the surrounding fluid's viscosity. When a gas bubble is immersed in a liquid, the dynamics of the bubble are influenced by both the properties of the gas and the liquid, as well as the interactions at their interface. 
The gas-liquid interface plays a crucial role in bubble dynamics. The behavior at this interface is influenced by surface tension and the pressure difference between the gas inside the bubble and the surrounding liquid. The movement and deformation of this interface can affect the stability and motion of the bubble.
The finite volume constraint adds an additional layer of complexity, as it implies that the bubble dynamics are affected by altering the pressure distribution and potential interactions with surfaces within the liquid. 

In the unbounded liquid scenario, where $\Omp(t)=\mathbb{R}^3\setminus \overline{\Omn(t)}$, Prosperetti \cite{Prosperetti1991} first considered the approximate model \eqref{eq-l-main}, \eqref{eq-g-main}, and \eqref{eq-b-main-1}--\eqref{eq-b-main-3} for the original problem \eqref{eq-l-full}--\eqref{eq-lg-full} (without external boundary conditions \eqref{eq-lo-full}). Later, Biro and Vel\'azquez \cite{Biro2000} proved the global existence of solutions in H\"{o}lder space for initial data near the spherically symmetric equilibria, as well as the Lyapunov stability of the equilibria under small mass-preserving perturbations, assuming that the liquid is inviscid on the gas-liquid free interface (i.e., $p_g \normal - p_l \normal=- \sigma \meancurv\normal$ in \eqref{eq-b-main-2}). They also considered the liquid pressure far away from the gas bubble, represented by an external forcing term $p_{\infty}(t)=\lim_{x\rightarrow\infty}p_{l}(x,t)$. Subsequently, Lai and Weinstein \cite{Lai2023} proved the exponentially asymptotic stability of the manifold of spherically symmetric equilibria, taking into account the liquid viscosity on the gas-liquid interface. They also demonstrated the existence and uniqueness of an exponentially asymptotically stable periodic spherically symmetric pulsating solution \cite{Lai2024}, given a small-amplitude, time-periodic $p_{\infty}(t)$.

When considering compressible liquids, Shapiro and Weinstein focused on the linearized problem of a gas bubble immersed in an inviscid compressible liquid. They demonstrated that the system exhibits exponential point-wise decay towards a family of equilibria \cite{Shapiro2011}. In the case of an unbounded liquid governed by compressible Navier-Stokes equations and a homogeneous spherical gas bubble following the polytropic gas law, Zhao and Zou proved the existence of global solutions and the asymptotic stability of the spherical equilibria in this free boundary problem \cite{Zhao2022}. 

We are not aware of any relevant rigorous theoretical studies for the situation in which the gas-liquid region is bounded, although there are some numerical studies available. Siegel conducted simulations in two dimensions, assuming that the gas motion is governed by the Stokes equations with a free boundary \cite{Siegel1999}. Lozinski and Romerio presented numerical results for the case that both the gas and the liquid are incompressible, and the entire gas-liquid system occupies a bounded, time-independent domain \cite{Lozinski2007}.   
\subsection*{Novelties and structure of the paper} 

The novelties of this study are as follows.
\begin{enumerate}[(i)]
\item We consider the scenario of a gas bubble immersed in a liquid with a finite volume. Unlike in an unbounded liquid, the presence of distinct boundaries in a finite liquid volume imposes additional constraints on the bubble's movement and interactions. Specifically, the bubble dynamics are influenced by a well-defined free boundary where surface tension plays a crucial role. In this bounded gas-liquid system, the steady-state gas-liquid region must assume a spherical shape.
For general surface tension coefficients $\sigma\neq\barsigma$, we analyze the unique positive real root of a ninth-order polynomial to establish that spherical equilibrium solutions correspond uniquely to specific pairs of gas mass and liquid volume. Additionally, we determine the upper and lower bounds for the steady-state radius of the gas bubble.

\item We analyze the energy dissipation for general solutions (not necessarily symmetric) to the free boundary problem \eqref{eq-l-main}--\eqref{eq-b-main}, which is important to the asymptotic stability analysis in general and extends the results from the unbounded liquid case presented in \cite[Proposition 7.4]{Lai2023}. Our results demonstrate that the equilibrium solution serves as a local minimizer of the energy functional even when subjected to relatively large perturbations. Specifically, this means that perturbations in the gas density can be accommodated up to a certain proportion of the equilibrium state, with the proportionality constant depending solely on the adiabatic constant $\Adi$. This result is notable because it shows that the extent of acceptable disturbance is independent of the liquid volume, making it particularly relevant for scenarios involving small liquid volumes.

\item When applying the central manifold theory, we construct general global central manifolds based on the algebraic equations of the steady-state solution parameterized by the mass-volume pair. This approach generalizes the local central manifold identified in \cite[Lemma 9.6]{Lai2023}.

\item For sufficiently small liquid volumes, we nearly achieve the optimal exponential decay rate. Additionally, we demonstrate that reducing the gas mass or increasing the temperature can accelerate the convergence --- an effect not observed with an infinite liquid volume. For spherically symmetric solutions, the convergence rate is solely determined by the spectrum of a linear operator, specifically  $\sup\{\operatorname{Re}(\lambda):\lambda\in\operatorname{sp}(\Linearpart)\setminus\{0\}\}$ (see Section \ref{s-exponential convergence} for details). By analyzing its upper bound and estimating the lower bound of a specific negative eigenvalue, we establish the existence of a constant $\Theta_0\in(0,1)$ such that
\begin{equation*}
-\pi^2\bkappa<\sup\{\operatorname{Re}(\lambda):\lambda\in\operatorname{sp}(\Linearpart)\setminus\{0\}\}\le-\Theta_0\pi^2\bkappa,
\end{equation*}
where $\pi^2\bkappa\sim \sqrt{\Tconst}/\sqrt{M}$, revealing a clear correlation between the exponential decay index, gas mass $M$, and external temperature $\Tconst$.
\end{enumerate}

The rest of the paper is structured as follows. In Section \ref{s:spherical}, we reduce free boundary problem \eqref{eq-l-main}--\eqref{eq-b-main} to system \eqref{eq-main-syst} and specify the equilibria under the spherical symmetry assumption. Section \ref{s:energy}  demonstrates the energy dissipation and shows that the steady-state solution acts as a local minimizer of the energy functional under proportional perturbations. In Section \ref{s:proof}, we establish the nonlinear and exponential asymptotic stability of the equilibria using center manifold theory, as detailed in Appendix \ref{appendix-4}. The calculation of equilibria for problem \eqref{eq-l-full}--\eqref{eq-lo-full} is provided in Appendix \ref{appendix-1}, while Appendix \ref{appendix-2} covers the well-posedness and Lyapunov stability of system \eqref{eq-main-syst}. Appendix \ref{appendix-3} includes necessary verifications for applying the center manifold theory. 
\section{Spherically symmetric solutions and equilibria}\label{s:spherical} 

In this section, we will show that, under the assumption of spherical symmetry, the free boundary problem \eqref{eq-l-main}--\eqref{eq-b-main} can be simplified to a system involving the gas density and the bubble radius, provided that the liquid volume and other parameters are fixed. Furthermore, we will demonstrate that the regular spherically symmetric equilibrium solution is uniquely determined by the gas mass and liquid volume. 

We assume that the velocity fields of the gas and the liquid are spherically symmetric (e.g., $\VV_g(x,t)=v_g(|x|,t)x/|x|,x\neq 0$), and the other scalar variables are radial (e.g., $\gdens(x,t)=\gdens(\module{x},t)$), to rewrite problem \eqref{eq-l-main}--\eqref{eq-b-main} as follows. 

\begin{proposition}\label{p-main-syst}
Given the liquid volume $V$, solving the regular spherically symmetric solution to system \eqref{eq-l-main}--\eqref{eq-b-main} reduces to an initial boundary value problem for the bubble radius $R(t)$ and the gas density $\gdens(r,t)=\gdens(x,t)$, where $r=\module{x}\le R(t)$. More precisely, for $r\le R(t)$ and $t>0$, we have 
\begin{subnumcases}{\label{eq-main-syst}}
\parl_t\gdens(r,t)=\frac{\Thermal}{\Adi\Heatcapa}\Laplace_r\log\gdens(r,t)+\frac{\dot{p}}{\Adi p(t)}\Paren{\frac{r\parl_r\gdens(r,t)}{3}+\gdens(r,t)},\label{eq-main-syst-1}\\
\ptR=-\frac\Thermal{\Adi\Heatcapa}\frac{\parl_r\gdens(R(t),t)}{\gdens(R(t),t)^2}-\frac{\dot{p}}{3\Adi p(t)}R(t),\label{eq-main-syst-2}\\
p(t)=\GASCONST\Tconst\gdens(R(t),t),\label{eq-main-syst-4}\\
\gdens(R(t),t)=\frac1{\GASCONST\Tconst}\left\lbrace4\mu_l\Paren{\frac{\ptR}{R(t)}-\frac{\ptbR}{\bR(t)}}+\frac{2\sigma}{R(t)}+\frac{2\barsigma}{\bR(t)}\right.\nonumber\\
\quad\quad\quad\quad\quad\quad\quad\quad  +\ldens\Bracket{\Paren{R(t)-\frac{R(t)^2}{\bR(t)}}\pttR+\Paren{\frac32-\frac{2R(t)}{\bR(t)}+\frac{R(t)^4}{2\bR(t)^4}} \ptR^2}\Bigg\rbrace,\label{eq-main-syst-3}  
\end{subnumcases} 
where $\bR=\sqrt[3]{R^3+(3V/4\pi)}$ represents the external radius of the entire gas-liquid region. The initial condition reads $(\gdens(\cdot,0),R(0),\ptR (0))=(\gdens_0,R_0,\ptR_0)$. 
Above, we omit the subscript for the gas pressure $p_g$, and denote the operator $\Laplace_r(\cdot)=r^{-2}\parl_r(r^2\parl_r(\cdot))$ with $r=\module{x}$. 
\end{proposition}  

\begin{proof} 
Recalling that the gas and gas-liquid regions are spheres $B_{R(t)}$ and $B_{\bR(t)}$ centered at the origin, respectively, we can assume that the spherically symmetric velocities are $ \VV_l(x,t)=v_l(r,t)x/r$ and $ \VV_g(x,t)=v_g(r,t)x/r$ for $x\neq0$, where $v_l$ and $v_g$ are scalar functions.  Additionally, $p_l,\gdens ,T_g$ and $s$ are radial functions.  Direct calculations yield
\begin{subnumcases}{} 
\ldens\parl_tv_l=\mu_l\Paren{\Laplace_rv_l-2v_lr^{-2}}-\ldens v_l\parl_rv_l-\parl_rp_l,&$R(t)\le r\le \bR(t)$,\label{eq-pro2-1-1}\\
r^{-2}\parl_r(r^2v_l)=0,&$R(t)\le r\le\bR(t)$,\label{eq-pro2-1-2}
\end{subnumcases}
\begin{subnumcases}{}  
\parl_t\gdens+r^{-2}\parl_r(\gdens r^2 v_g)=0,&$r\le R(t)$,\\
\dot{p}p^{-1}=\Thermal\Heatcapa^{-1}r^{-2}\parl_r\Paren{r^2\parl_r\Paren{\gdens^{-1}}}-\Adi r^{-2}\parl_r(r^2v_g),&$r\le R(t)$, 
\end{subnumcases}  
and the boundary conditions \eqref{eq-b-main} become
\begin{subnumcases}{\label{eq-pro2-2}} 
v_l(R(t),t)=v_g(R(t),t)=\ptR(t),\quad v_l(\bR(t),t)=\ptbR(t),\label{eq-pro2-2-1}\\
p(t)-p_l(R(t),t)+2\mu_l\parl_rv_l(R(t),t)=2\sigma R(t)^{-1},\\  
-p_l(\bR(t),t)+2\mu_l\parl_rv_l(\bR(t),t)=-2\barsigma\bR(t)^{-1},\\
T(R(t),t)=\Tconst,\label{eq-pro2-2-4} 
\end{subnumcases} 
where we have used the fact that the curvature of a sphere $B_R$ is $-2R^{-1}$. 

Clearly that \eqref{eq-main-syst-4} follows from \eqref{eq-l-main-3}, \eqref{eq-g-main-2}, \eqref{eq-g-main-5} and \eqref{eq-b-main-3}. Then, the divergence-free condition \eqref{eq-pro2-1-2} and  kinematic boundary condition \eqref{eq-pro2-2-1} imply that for $t>0$, 
\begin{equation}\label{eq-pro2-3}
v_l(r,t)=R(t)^2\ptR(t)r^{-2}=\bR(t)^2\ptbR(t)r^{-2},\quad R(t)\le r\le\bR(t).
\end{equation} 
We also note that $R^2\ptR=\bR^2\ptbR$.  
Omitting the variable $t$ and substituting \eqref{eq-pro2-3} into  \eqref{eq-pro2-1-1}, we have $(2R\ptR^2+R^2\pttR)r^{-2}=2R^4\ptR^2r^{-5}-\ldens^{-1}\parl_rp_l$.
By integrating, we deduce that for $R(t)\le r\le\bR(t)$ and $t>0$,
\begin{equation*}
p_l(r,t)-p_l(R,t)
=\ldens\Bracket{R^4\ptR^2\Paren{R^{-4}-r^{-4}}/2-\Paren{2R\ptR^2+R^2\pttR}\Paren{R^{-1}-r^{-1}}}.
\end{equation*}
Moreover, \eqref{eq-pro2-3} implies $\parl_rv_l(r,t) = -2R^2\ptR r^{-3}=-2\bR^2\ptbR r^{-3}$. As a result,  $\parl_rv_l(R(t),t)=-2\ptR R^{-1}$ and $\parl_rv_l(\bR(t),t)=-2\bR^{-1}\ptbR$. Therefore, we obtain
$p(t)-p_l(R(t),t)-4\mu_l\ptR R^{-1}=2\sigma R^{-1}$ and $-p_l(\bR(t),t)-4\mu_l\bR^{-1}\ptbR=-2\barsigma\bR^{-1}$. Combining these calculations, it follows that
\begin{equation*} 
p_l(r,t)=p(t)-4\mu_l\ptR R^{-1}-2\sigma R^{-1}+\ldens\Bracket{\frac{1}{2}\ptR^2-\frac{1}{2}R^4\ptR^2r^{-4}-\frac{r-R}{r}\Paren{2\ptR^2+R\pttR}},
\end{equation*}
where $R\le r\le\bR$. Setting $r=\bR$, one has
\begin{equation}\label{eq-pro2-4}
p(t)=4\mu_l \Paren{\frac{\ptR }{R}- \frac{\ptbR }{\bR}}+\frac{2\sigma}{R}+\frac{2\barsigma}{\bR}-\ldens \Bracket{\frac {\ptR^2}{2}-\frac {R^4\ptR^2}{2\bR^4}-\frac{\bR-R}{\bR}\Paren{2\ptR^2 + R \pttR}}.
\end{equation}

From the gas dynamics in \eqref {eq-g-main-equivalent} and \eqref{eq-parl-gas density}, for $0\le r\le R(t)$, we have 
\begin{equation}\label{eq-pro2-5} 
v_g=\frac{\Thermal}{\Adi\Heatcapa}\parl_r\Paren{\frac{1}{\gdens}}-\frac{r}{3\Adi}\frac{\dot{p}}{p}\ \text{and}\ \parl_t\gdens=\frac\Thermal{\Adi\Heatcapa}\Laplace_r\log\gdens+\frac{\dot{p}}{3\Adi p}r\parl_r\gdens+\frac{\dot{p}}{\Adi p}\gdens.
\end{equation}
Thus, equation \eqref{eq-main-syst-1} follows. Taking the time derivative of both sides of \eqref{eq-main-syst-4}, we obtain $\dot{p}p^{-1}=\parl_t\gdens(R,t)\gdens(R,t)^{-1}+\ptR\parl_r\gdens(R,t)\gdens(R,t)^{-1}$. Evaluating \eqref{eq-pro2-5} at $r=R(t)$ and using the kinematic boundary condition \eqref{eq-pro2-2-1}, we derive \eqref{eq-main-syst-2}. Finally, combining \eqref{eq-main-syst-4}  and \eqref{eq-pro2-4} leads to the boundary condition \eqref{eq-main-syst-3}. This completes the proof.   
\end{proof}  

We note that a regular solution $(\gdens(r,t),R(t))$ to system \eqref{eq-main-syst} corresponds a spherically symmetric solution $(v_{l},p_l,\gdens,v_{g},p_g,T_g,s)$ to system \eqref{eq-l-main}--\eqref{eq-b-main}, where the velocity $v_{l}$ is given by formula \eqref{eq-pro2-3}.

In the following  proof and the rest of the paper, the external radius will always be denoted by 
\begin{equation*} 
\bR=\sqrt[3]{R^3+\bV},
\end{equation*} 
where $\bV=3V/4\pi$ represents the modified liquid volume and $R$ denotes a generic radius of the gas bubble. 

\begin{proof}[Proof of Theorem \ref {t-equilibria}]
According to Proposition \ref{p-main-syst}, it suffices to compute the steady solutions to system \eqref{eq-main-syst}. Setting $\parl_t\gdens=\ptR=0$,
we obtain from \eqref{eq-main-syst-1} and \eqref{eq-main-syst-2} that $R(t)\equiv\ssr$ (equilibrium gas radius), $\Laplace\log\gdens=0$ in $B_{\ssr}$, and $\parl_r\gdens(\ssr)=0$.
Then it follows that $\gdens(r)\equiv\ssd$ (equilibrium gas density). 
This, combined with  \eqref{eq-main-syst-3} yields $\ssd=\frac{2}{\GASCONST\Tconst}\Paren{\frac{\sigma}{\ssr}+\frac{\barsigma}{\ssor}}$.
Moreover, the conservative mass of the gas can be expressed by the pair $(\ssd,\ssr)$, i.e., $M=\frac{4\pi}{3}\ssd\ssr^3$.
Therefore, for any mass-volume pair $(M,V)\in(0,\infty)^2$, the equilibrium $(\ssd,\ssr)$ is determined by the following
algebraic equations: 
\begin{equation}\label{eq-equi-eq}
\frac{4\pi}{3}\ssd\ssr^3=M\ \text{and}\
\ssd=\frac{2}{\GASCONST\Tconst}\Paren{\frac{\sigma}{\ssr}+\frac{\barsigma}{\ssor}}. 
\end{equation} 
We denote 
\begin{equation*}
I=\frac{3\GASCONST\Tconst M}{8\pi\sigma}\ \text{and}\ \surftenratio=\frac{\barsigma}{\sigma}>0.
\end{equation*} 
Then, \eqref{eq-equi-eq} is equivalent to
\begin{equation}\label{eq-equi-eq-2} 
\frac{1}{\ssr}+\frac{\surftenratio}{\ssor}=\frac{I}{\ssr^3}\ \text{or}\ \frac{\surftenratio^3}{\ssr^3+\bV}=\Paren{\frac{I-\ssr^2}{\ssr^3}}^3,
\end{equation}  
where $\bV=3V/4\pi$.
The equilibrium radius $\ssr$ (if it exists) is exactly a positive real root to the polynomial
\begin{equation}\label{eq-9-poly}
\mathbb{P}(x)=(\surftenratio^{3}+1)x^{9}-3Ix^{7}+\bV x^{6}+3I^{2}x^{5}-3I\bV x^{4}-I^{3}x^{3}+3I^{2}\bV x^{2}-I^{3}\bV.
\end{equation}
At the same time, $\ssr$  satisfies $\ssr\in(\sqrt{I}/\sqrt{1+\surftenratio},\sqrt{I})$ by using \eqref{eq-equi-eq-2}. 

From these observations, we introduce a function $\mathbb{K}(w)=\mathbb{P}(w\sqrt{I})$, where the variable $w\in [1/\sqrt{1+\surftenratio},1]$, and a direct calculation shows
\begin{equation*}
\mathbb{K}(w)=-\Bracket{\Paren{I^{-\frac 12} w^{-3}-I^{-\frac 12}w^{-1}}^{3}\Paren{I^{\frac{3}{2}}w^{3}+\bV}-\surftenratio^{3}}I^{\frac{9}{2}}w^{9}\triangleq-\mathbb{L}(w)I^{\frac{9}{2}}w^{9}.
\end{equation*}
We note that the function $\mathbb{L}(w)$ is strictly decreasing, since for $w\in \Paren{1/\sqrt{1+\surftenratio},1}$, it holds
\begin{equation*}
\mathbb{L}^\prime(w)=3I^{-\frac{3}{2}}\Bracket{-2I^{\frac{3}{2}}w^{3}-\bV(3-w^2)}\Paren{w^{2}-1}^{2}w^{-10}<0. 
\end{equation*}
Furthermore, we have 
\begin{equation*} \mathbb{L}(1/\sqrt{1+\surftenratio})=\bV I^{-\frac{3}{2}}(1+\surftenratio)^{\frac 32}\surftenratio^3>0\ \text{and}\ \mathbb{L}(1)=-\surftenratio^3<0.
\end{equation*}
In the above, we emphasize that $\mathbb{L}(1/\sqrt{1+\surftenratio})$ is positive regardless of the modified volume $\bV$. 

We conclude that $\mathbb{L}$ has a unique positive solution in $(1/\sqrt{1+\surftenratio},1)$. Therefore, the same result holds for $\mathbb{K}(w)$. Coming back to $\mathbb{P}(x)$, the ninth-degree polynomial, it possesses a unique positive root $\ssr$ within $(\sqrt{I}/\sqrt{1+\surftenratio},\sqrt{I})$, which is uniquely determined by  $M$ and $V$, with the other parameters being constants. Thus, the map $\ssr[M,V]$ is well-defined and satisfies
\begin{equation}\label{eq-Rstar-bound}
\sqrt{\frac{3\GASCONST\Tconst M}{8\pi\sigma}}\bigg/\sqrt{1+\frac{\barsigma}{\sigma}}<\ssr[M,V]<\sqrt{\frac{3\GASCONST\Tconst M}{8\pi\sigma}}.
\end{equation}
The smoothness of the map $\ssr[M,V]$ is the consequence of the smooth dependence of a simple root to the polynomial $\mathbb{P}$ on its coefficients.

Once we determine the equilibrium gas radius $\ssr[M,V]$, the equilibrium density is expressed as
\begin{equation}\label{eq-rhostar}
\ssd[M,V]=\frac{3M}{4\pi\ssr[M,V]^3}=\frac2{\GASCONST\Tconst}\Paren{\frac{\sigma}{\ssr[M,V]}+\frac{\barsigma}{\ssor[M,V]}},
\end{equation} 
and the equilibrium gas pressure $p$ follows by using \eqref{eq-main-syst-4}. Then, we can recover the corresponding steady state solution to system \eqref{eq-l-main}--\eqref{eq-b-main}. That is, $\VV_{g}=\VV_{l}\equiv\Zero$, $T_{g}\equiv\Tconst$, $s\equiv\Heatcapa\log (p/\ssd^\Adi)$, and $p_{l}\equiv2\barsigma/\ssor$. 

In turn, given any equilibrium state $(\ssd,\ssr)$, we can specify the mass of the gas and liquid volume by \eqref{eq-equi-eq-2} since the coefficients of both $\bV$ and $(\ssr^2-I)^3$ are nonzero. More precisely, there exists a one-to-one correspondence between $(\ssd,\ssr)$ and $(M,V)$
\begin{equation}\label{eq-one-to-one}
\Paren{M,V}=\frac{4\pi\ssr[M,V]^3}{3}\Paren{\ssd[M,V],\Bracket{\frac{2\sigma^2}{\Paren{\GASCONST\Tconst \ssd[M,V]\ssr[M,V]-2\sigma}\barsigma}}^3-1}.  
\end{equation} 
This completes the proof.  
\end{proof}

\begin{remark}
If there is no surface tension on the external free boundary (i.e., $\barsigma=0$), equations \eqref{eq-equi-eq} reduce to a quadratic function. In this case, the equilibrium radius simplifies to $\ssr=\sqrt{\frac{3\GASCONST\Tconst M}{8\pi\sigma}}$, which is independent of the liquid volume. Consequently, \eqref{eq-one-to-one} no longer holds. Nevertheless, $\ssr$ remains the unique positive root of \eqref{eq-9-poly} and corresponds to the critical case of \eqref{eq-Rstar-bound}.
\end{remark}

The local well-posedness, global well-posedness, and the Lyapunov stability for system \eqref{eq-main-syst} are detailed in Appendix \ref{appendix-2}. In the next section, after introducing the physical energy, we will prove the energy dissipation and show that any equilibrium solution acts as a local minimizer of the energy functional.
\section{Energy dissipation and local minimizers of the energy functional}\label{s:energy}

The physical energy of a general solution (without any symmetric assumption) to the approximate model \eqref{eq-l-main}--\eqref{eq-b-main} is given by $\Energy(t)=\Energy_1(t)+\Energy_2(t)$. The energy $\Energy_1(t)$ denotes the Helmholtz free energy (see, e.g., \cite{Landau2013})
\begin{equation*}
\Energy_1(t)=\Heatcapa\int_{\Omn(t)}\gdens T_gdx-\Tconst\int_{\Omn(t)}\gdens sdx,
\end{equation*} 
The energy $\Energy_2(t)$ includes both the kinetic energy of the liquid and the surface energy, encompassing the gas-liquid interface as well as the external free boundary of the gas-liquid region  
\begin{equation*}
\Energy_2(t)=\frac12\int_{\Omp(t)}\ldens|\VV_l|^2dx+\sigma\int_{\parl\Omn(t)}dS+\barsigma\int_{\parl\Omp(t)\setminus\parl\Omn(t)}dS.
\end{equation*} 
\subsection{Energy dissipation law for general solutions to problem \eqref{eq-l-main}--\eqref{eq-b-main}}\label{s:energy dissipation} 

Recalling that the liquid volume is finite and $\overline{\Omn(t)\cup\Omp(t)}$ is bounded, we compute the associated energy dissipation.
\begin{proposition} 
\label{p-energy-dissipation}
For  free boundary problem \eqref{eq-l-main}--\eqref{eq-b-main}, we have the following energy dissipation law:
\begin{equation}\label{eq-energy-dissipation}
\frac{d}{dt}\Energy(t)=-\frac{\mu_l}{2}\int_{\Omp(t)}\module{\nabla\VV_l+\nabla\VV_l^\top}^2dx-\Thermal\Tconst\int_{\Omn(t)}\frac{\module{\nabla\gdens}^2}{\gdens^2}dx.
\end{equation} 
\end{proposition}
\begin{proof} 
Given that boundary conditions \eqref{eq-b-main-1} and \eqref{eq-b-main-4} are satisfied, we recall the transport formulas for the time-dependent regions
\begin{align*}
\frac{d}{dt}\int_{\Omn(t)}fdx&=\int_{\Omn(t)}\Paren{\parl_t+\VV_g\cdot\nabla}f+\Divergence\VV_g fdx,\\
\frac{d}{dt}\int_{\Omp(t)}fdx&=\int_{\Omp(t)}\Paren{\parl_t+\VV_l\cdot\nabla}fdx,
\end{align*} 
and the analogous formula on the liquid moving surface $\parl\Omp(t)$, which encompasses both the  gas-liquid interface $\parl\Omn(t)$ and the external free boundary $\parl\Omp(t)\setminus\parl\Omn(t)$ (see, e.g., \cite{Shatah2008a})
\begin{equation*}
\frac{d}{dt}\int_{\parl\Omp(t)}fdS=\int_{\parl\Omp(t)}\Paren{\parl_t+\VV_l\cdot\nabla}f+\tangentdiv \VV_lfdS.
\end{equation*} 
These,  along with \eqref{eq-l-main-1}, \eqref{eq-b-main-2}, \eqref{eq-b-main-5}, integration by parts formula, and the divergence theorem on the moving surface, yield
\begin{align*}
\dot\Energy_2(t)
={}&\ldens\int_{\Omp(t)}\Paren{\parl_t\VV_l+\VV_l\cdot\nabla\VV_l}\cdot\VV_ldx+\sigma\int_{\parl\Omn(t)}\tangentdiv \VV_ldS+\barsigma\int_{\parl\Omp(t)\setminus\parl\Omn(t)}\tangentdiv \VV_ldS\\
={}&\mu_l\int_{\Omp(t)}\Laplace\VV_{l,i}\VV_l^idx-\int_{\Omp(t)}\parl_ip_l\VV_l^idx+\sigma\int_{\parl\Omn(t)}\meancurv\Bracket{\VV_l\cdot\Paren{-\normal}}dS\\
&+\barsigma\int_{\parl\Omp(t)\setminus\parl\Omn(t)}\meancurv \Paren{\VV_l\cdot\normal}dS\\
={}&-\mu_l\int_{\Omp(t)}|\nabla\VV_l|^2dx+\mu_l\int_{\parl\Omn(t)}(-\normal^j)\parl_j\VV_{l,i}\VV_l^idS\\
&+\mu_l\int_{\parl\Omp(t)\setminus\parl\Omn(t)}\normal^j\parl_j\VV_{l,i}\VV_l^idS-\int_{\parl\Omn(t)}-\normal_ip_l\VV_l^idS-\int_{\parl\Omp(t)\setminus\parl\Omn(t)}\normal_ip_l\VV_l^idS\\
&-\sigma\int_{\parl\Omn(t)}\meancurv \Paren{\VV_l\cdot\normal}dS+\barsigma\int_{\parl\Omp(t)\setminus\parl\Omn(t)}\meancurv\Paren{\VV_l\cdot\normal}dS\\
={}&-\mu_l\int_{\Omp(t)}|\nabla\VV_l|^2dx+\int_{\parl\Omn(t)}\Paren{p_l\normal-\mu_l\normal^j\parl_j\VV_{l}-\sigma\meancurv\normal}\cdot\VV_l dS\\
&-\int_{\parl\Omp(t)\setminus\parl\Omn(t)} \Paren{p_l\normal-\mu_l \normal^j\parl_j\VV_{l}+\barsigma\meancurv}\cdot\normal dS\\
={}&-\mu_l\int_{\Omp(t)}|\nabla\VV_l|^2dx+p_g\int_{\parl\Omn(t)}\VV_l\cdot\normal dS+\mu_l\int_{\parl\Omn(t)}\parl_j\VV_l^k\VV_{l}^j\normal_kdS\\
&-\mu_l\int_{\parl\Omp(t)\setminus\parl\Omn(t)}\parl_j\VV_l^k\VV_{l}^j\normal_kdS\\
={}&-\mu_l\int_{\Omp(t)}\nabla\VV_l:\Paren{\nabla\VV_l+\nabla\VV_l^\top}dx+p_g\int_{\parl\Omn(t)}\VV_l\cdot\normal dS.
\end{align*}
For the energy $\Energy_1(t)$, we apply the transport formula, gas system \eqref{eq-g-main}, boundary conditions \eqref{eq-b-main-1} and \eqref{eq-b-main-3} to obtain
\begin{align*}
\dot\Energy_1(t)
={}&\frac{\Heatcapa}{\GASCONST}\dot p_g\int_{\Omn(t)}1dx+\frac{\Heatcapa}{\GASCONST}p_g\int_{\Omn(t)}\Divergence\VV_gdx-\Tconst\int_{\Omn(t)}\gdens \Paren{\parl_ts+\VV_g\cdot\nabla s}dx\\
={}&\frac{\dot p_g}{\Adi-1}\int_{\Omn(t)}1dx+\frac{p_g}{\Adi-1}\int_{\Omn(t)}\Divergence\VV_gdx-\Tconst\Thermal\int_{\Omn(t)}\frac{\Laplace T_g}{T_g}dx\\
={}&\frac{\dot p_g}{\Adi-1}\int_{\Omn(t)}1dx+\frac{p_g}{\Adi-1}\int_{\parl\Omn(t)}\VV_g\cdot\normal dS\\
&-(-1)\Tconst\Thermal\int_{\Omn(t)}-\frac{\module{\nabla T_g}^2}{T_g^2}dx-\Tconst\Thermal\int_{\parl\Omn(t)}\frac{\nabla T_g\cdot\normal}{T_g}dS\\
={}&\frac{\dot p_g}{\Adi-1}\int_{\Omn(t)}1dx-\Tconst\Thermal\int_{\Omn(t)}\frac{\module{\nabla\gdens}^2}{\gdens^2}dx+\frac{p_g}{\Adi-1}\int_{\parl\Omn(t)}\VV_l\cdot\normal dS\\
&-\Thermal\int_{\parl\Omn(t)}\nabla T_g\cdot\normal dS,
\end{align*}
where we have used the relation $\Adi=1 + {\GASCONST}/{\Heatcapa}$.
Noting that
\begin{equation*}
\Thermal\int_{\parl\Omn(t)}\nabla T_g\cdot\normal dS=\Thermal\int_{\Omn(t)}\Laplace T_gdx=\frac{p_g}{\GASCONST}\int_{\Omn(t)}\Thermal\Laplace\Paren{\gdens^{-1}}dx,
\end{equation*}
and recalling \eqref{eq-gas-velocity}, we deduce
\begin{align*}
\Thermal\int_{\parl\Omn(t)}\nabla T_g\cdot\normal dS&=\frac{\Heatcapa}{\GASCONST}\dot p_g\int_{\Omn(t)}1dx+\frac{\Heatcapa\Adi}{\GASCONST}p_g\int_{\Omn(t)}\Divergence\VV_gdx\\
&=\frac{\dot p_g}{\Adi-1}\int_{\Omn(t)}1dx+\frac{\Adi}{\Adi-1}p_g\int_{\parl\Omn(t)}\VV_l\cdot\normal dS.
\end{align*}
Collecting the above calculations, energy dissipation \eqref{eq-energy-dissipation} follows since the algebraic identity $\nabla\VV_l:\Paren{\nabla\VV_l+\nabla\VV_l^\top}=\module{\nabla\VV_l+\nabla\VV_l^\top}^2/2$ holds. 
\end{proof}

\subsection{Equilibria as local minimizers of the energy functional relative to proportional perturbations}\label{s:local minimizer} 

To establish the nonlinear stability, we will show that any equilibrium density-radius pair $(\ssd[M,V],\ssr[M,V])$, determined by the mass-volume pair $(M,V)$, serves as a local minimizer of the energy functional $E(t)$ with respect to mass-persevering and volume-invariant perturbations.

For this problem, in the case of an infinite liquid volume and an external forcing term $p_{\infty}(t)$,  Biro and Vel\'azquez \cite[Lemma 4.2]{Biro2000} used Taylor's theorem with Peano's remainder form to expand the total energy at the equilibrium state up to quadratic terms and derived the coercivity energy estimate. Based on this estimate, the Lyapunov stability was established \cite[Theorem 4.1]{Biro2000}. More recently, Lai and Weinstein \cite[Theorem 7.5]{Lai2023} extended these results to scenarios with more general far-field pressure conditions. The key coercive energy estimate \cite[(7.17)]{Lai2023} was obtained under sufficiently small perturbations, which depend on the mass of the gas bubble.

We extend the results to the case where the liquid volume $V$ is finite. This is a significant improvement as we have removed the assumptions, regarding the upper and lower bounds of the gas density and the bubble radius, made in previous studies (such as $C^{-1}\le \gdens,R\le C$ in \cite{Biro2000,Lai2023}), and now allow for relatively larger perturbations, as long as they are within a certain proportion of the equilibrium density (see \eqref{eq-minimizer-condition}). It should also be noted that the constant $\delta_0$ in \eqref{eq-minimizer-condition} is independent of the constants $\Tconst,\ldens$, the equilibrium state $(\ssd,\ssr)$, the mass $M$, and the volume $V$. This is a crucial result that enables us to prove the main stability results, specifically \eqref{eq-main- results-i} in Theorem \ref{t-main result}. 

The strategy we have adopted is as follows. We first consider a specific proportion of disturbance to the equilibrium density. Since the mass of the gas and the volume of the liquid remain constant, the amplitude of the disturbance in the bubble radius will be determined by the change in gas density. Next, we utilize Taylor's theorem to approximate the energy functional near the equilibrium state. We then apply the Lagrangian form of the remainder, which provides the most precise expression for the nonlinear terms. This is because the bulk terms $I_7, I_8$, and $I_9$ (in the second-order derivative below) are related to the liquid volume. Peano's form of the remainder cannot provide the necessary control, as the liquid volume can be arbitrarily small.

\begin{theorem}\label{t-minimizer}
There exists a constant $\delta_0>0$ depending only on $\Adi$, such that the following holds: given any equilibrium $(\ssd,\ssr)=(\ssd[M,V],\ssr[M,V])$ with mass $M$ and liquid volume $V$, for any mass-persevering and volume-invariant spherically symmetric state $(\gdens(r,t),R(t),\bR(t)=\Paren{R(t)^3+\bV}^{1/3})$, such that 
\begin{equation}\label{eq-minimizer-condition}
\norm{\frac{\gdens(\cdot,t)-\ssd}{\ssd}\Paren{1+\module{\log\ssd}}}_{L^\infty(B_{R(t)})}\le\delta_0,
\end{equation}
then, we have
\begin{align}
\Energy(t)\ge{}&\Energy_{\dagger}+\frac{1}{4}\left\lbrace M\Heatcapa\Tconst\Paren{\frac{\brho(1,t)-\ssd}{\ssd}-\frac{3}{4\pi\ssd}\int_{B_{1}}\brho(y,t)-\ssd dy}^2\right.\nonumber\\
&\qquad\quad\ \ +\frac{\ldens\ssr^5}{4\pi\ssd^2}\Paren{1-\frac{\ssr}{\ssor}}\Paren{\int_{B_{1}}\dot\brho(y,t)dy}^2+\frac{\GASCONST\Tconst\ssr^3}{3\ssd}\int_{B_{1}}\module{\brho(y,t)-\ssd}^2dy\nonumber\\
&\qquad\quad\ \ +\left.\frac{\ssr^3}{\pi\ssd^2}\Bracket{\frac{\sigma}{2\ssr}+\frac{\barsigma}{\ssor}\Paren{1-\frac{\ssr^3}{2\ssor^3}}}\Paren{\int_{B_{1}}\brho(y,t)-\ssd dy}^2\right\rbrace,\label{eq-minimizer}
\end{align}
where $\Energy(t)=\Energy[\brho(\cdot,t),R(t)]$ with $\brho(y,t)=\gdens(R(t)|y|,t)$, $\Energy_{\dagger}=\Energy[\ssd,\ssr]$, and $t>0$ is arbitrary. Moreover, $(\gdens(r,t),R(t))$ is not restricted to be a solution of system \eqref{eq-main-syst}.
\end{theorem}  
\begin{proof}    
Setting $x=R(t)y$ for $y\in \overline{B_1}$, the domain $\overline{B_{R(t)}}$ is fixed
to $\overline{B_1}$. Defining  $\brho(|y|,t)=\gdens(R(t)|y|,t)$ and using \eqref{eq-g-main-4}, we have $p=\GASCONST\Tconst\brho(1,t)$. Then, the energy $\Energy_1(t)$ becomes 
\begin{align*}
\Energy_1(t)&=\frac{4\pi\Heatcapa}{3\GASCONST}p_gR^3-\Heatcapa\Tconst\GASMASS[\gdens,R]\log p_g+\Heatcapa\Adi\Tconst\int_{B_R}\gdens\log\gdens dx\\
&=\frac{4\pi\Heatcapa\Tconst}3\brho(1,t)R(t)^3-\Heatcapa\Tconst M\log\Bracket{\GASCONST\Tconst\brho(1,t)}+\Heatcapa\Adi\Tconst R(t)^3\!\!\int_{B_1}\!\!\brho(x,t)\log\brho(x,t)dx.
\end{align*}
Combined with the energy $\Energy_2(t)$ in the spherically symmetric scenario, i.e.,
\begin{equation*}
\Energy_2(t)=2\pi\ldens R^3\ptR^2-2\pi\ldens R^4\ptR^2/\bR+4\pi\sigma R(t)^2+4\pi\barsigma\bR(t)^2,
\end{equation*} 
we obtain that $\Energy(t)=\Energy[\brho(x,t),R(t)]$ is a functional of $(\brho(x,t),R(t))$ since $\bR(t)$ can be expressed by $R(t)$ and $V$
\begin{align*}
\Energy(t)={}&\frac{4\pi\Heatcapa\Tconst}3\brho(1,t)R^3-\Heatcapa\Tconst M\log(\GASCONST\Tconst)-\Heatcapa\Tconst M\log\brho(1,t)+\Heatcapa\Adi\Tconst R^3\int_{B_1}\brho\log\brho\\
&+2\pi\ldens R^3\ptR^2-2\pi\ldens R^4\ptR^2/\bR+4\pi(\sigma R^2+\barsigma\bR^2),
\end{align*} 
where we omit $dx$ and abbreviate the notations $R=R(t),\brho=\brho(x,t)$, etc.

Fix any $t>0$, we set $\brho_{\delta}(\cdot,t)=\ssd+\delta\Errorrho(\cdot,t)$ where $0\le\delta\le 1$ and $\norm{\Errorrho(\cdot,t)}_{L^\infty(B_1)}\le\ssd/2$. Then, there exists a unique $R_\delta=R_{\delta}(t)>0$ such that $R_{\delta}(t)^3\int_{B_{1}}\brho_{\delta}(y,t)=M$, where $M=\frac{4\pi}{3}\ssd\ssr^3$.  Define $\gdens_\delta(x,t)=\brho_\delta(x/R_\delta,t)$, and we have
\begin{equation}\label{eq-mini-0}
\norm{\gdens_\delta(\cdot,t)-\ssd}_{L^\infty(B_{R_\delta(t)})}\le{\delta\ssd}/{2}.
\end{equation}  
Also, $(\gdens_\delta,R_\delta)$ satisfies $\GASMASS[\gdens_\delta(\cdot,t),R_\delta(t)]=M$, i.e.,
\begin{equation*}
\frac{4\pi\ssd}{3}\ssr^3=M=R_{\delta}(t)^3\int_{B_{1}}\Paren{\ssd+\delta\Errorrho(y,t)}=R_{\delta}(t)^3\Paren{\frac{4\pi\ssd}{3}+\delta\int_{B_{1}}\Errorrho(y,t)}.
\end{equation*}
Then, we have
\begin{equation}\label{eq-mini-1}
R_{\delta}(t)^3
=\frac{\ssr^3}{1+\delta\phi(t)}\ \text{where}\ \phi(t)=\frac{3}{4\pi\ssd}\int_{B_{1}}\Errorrho(y,t).
\end{equation}
Also, it holds  
\begin{equation*}
\dot R_\delta(t)=\frac{-\delta\ssr\dot\phi(t)}{3\Paren{1+\delta\phi(t)}^{\frac 43}}\ \text{and}\ \bar{R}_\delta(t)=\Paren{\frac{\ssr^3}{1+\delta \phi(t)}+\bV}^{\frac 13}.
\end{equation*}
Next, we define $f(\delta)=f(\delta,t)=\Energy[\ssd+\delta\Errorrho(\cdot,t),R_\delta(t),]$ for $\delta\in[0,1]$ and it holds 
\begin{align*}
f(\delta)
={}&\frac{4\pi \Heatcapa\Tconst}3(\ssd+\delta\Errorrho(1,t))\frac{\ssr^3}{1+\delta \phi(t)}-\Heatcapa\Tconst M\log(\GASCONST\Tconst)-\Heatcapa\Tconst M\log(\ssd+\delta\Errorrho(1,t))\\
&+\frac{\Heatcapa\Adi\Tconst\ssr^3}{1+\delta\phi(t)}\int_{B_1}(\ssd+\delta\Errorrho)\log(\ssd+\delta\Errorrho)+2\pi\ldens\frac{\ssr^3}{1+\delta\phi(t)}\Paren{\frac{-\delta\ssr\dot\phi(t)}{3\Paren{1+\delta\phi(t)}^{\frac 43}}}^2\\
&-2\pi\ldens\Paren{\frac{\ssr}{\Paren{1+\delta\phi(t)}^{\frac 13}}}^4 \Paren{\frac{-\delta\ssr\dot\phi(t)}{3\Paren{1+\delta\phi(t)}^{\frac 43}}}^2\Paren{\frac{\ssr^3}{1+\delta\phi(t)}+\bV}^{-\frac 13}\\
&+4\pi\Bracket{\sigma\Paren{\frac{\ssr}{\Paren{1+\delta\phi(t)}^{\frac13}}}^2+\barsigma\Paren{\frac{\ssr^3}{1+\delta\phi(t)}+\bV}^{\frac23}}.
\end{align*} 

From Taylor's theorem with Lagrangian remainder, we have $f(\delta)=f(0)+f^\prime(0)\delta+{f^{''}(0)}\delta^2/2+{f^{'''}(\xi(\delta))\delta^3}/{6}$, where $0\le\xi(\delta)\le\delta$.  Direct calculations show that
\begin{align*}
f^\prime(0)={}&\frac{4\pi\ssr^3}{3}\Heatcapa\Tconst\Errorrho(1,t)-\frac{4\pi\ssr^3}3\Heatcapa \Tconst\ssd\phi-\frac{4\pi\ssr^3}{3}\Heatcapa\Tconst\Errorrho(1,t)-\Heatcapa\Adi\Tconst\ssr^3\phi\int_{B_1}\ssd\log\ssd\\
&+\Heatcapa\Adi\Tconst\ssr^3\int_{B_1}\Paren{\Errorrho\log(\ssd)+\Errorrho}-\frac{8}{3}\pi\Bracket{\sigma\ssr^2{\phi}+\barsigma{\ssr^3\phi}\Paren{{\ssr^3}+\bV}^{-\frac 13}}\\
={}&-\frac{4\pi\ssr^3}3\Heatcapa \Tconst\ssd\frac{3}{4\pi\ssd}\int_{B_{1}}\Errorrho-\Heatcapa\Adi\Tconst\frac{4\pi\ssr^3}{3}\frac{3}{4\pi\ssd}\int_{B_{1}}\Errorrho\ssd\log\ssd+\Heatcapa\Adi\Tconst\ssr^3\int_{B_1}\Errorrho\\
&+\Heatcapa\Adi\Tconst\ssr^3\log\ssd\int_{B_1}\Errorrho-\frac{4\pi}{3}\ssr^3\Paren{2\sigma/\ssr+2\barsigma/\ssor}\frac{3}{4\pi\ssd}\int_{B_{1}}\Errorrho\\
={}&(\Adi-1)\Heatcapa\Tconst\ssr^3\int_{B_{1}}\Errorrho-\ssr^3\ssd^{-1}\Paren{2\sigma/\ssr+2\barsigma/\ssor}\int_{B_{1}}\Errorrho\\
={}&\Bracket{\GASCONST\Tconst-\ssd^{-1}\Paren{2\sigma/\ssr+2\barsigma/\ssor}}\ssr^3\int_{B_{1}}\Errorrho=0,
\end{align*}
where we have used $\Adi=1+\GASCONST/\Heatcapa$ and \eqref{eq-equi-eq}. For $f^{''}(0)$, one has
\begin{align*}
f^{''}(\delta)={}&-\frac{2M\Heatcapa\Tconst\Errorrho(1,t)}{\ssd(1+\delta\phi)^2}\phi+\frac{2M\Heatcapa\Tconst}{\ssd}\frac{\ssd+\delta\Errorrho(1,t)}{\Paren{1+\delta\phi}^3}\phi^2+\frac{M\Heatcapa\Tconst\Errorrho(1,t)^2}{\Paren{\ssd+\delta\Errorrho(1,t)}^2} \\
&+\Paren{\frac{2\Heatcapa\Adi\Tconst\ssr^3}{\Paren{1+\delta\phi}^3}\int_{B_1}(\ssd+\delta\Errorrho)\log(\ssd+\delta\Errorrho)}\phi^2\\
&-\Paren{\frac{2\Heatcapa\Adi\Tconst\ssr^3}{\Paren{1+\delta\phi}^2}\int_{B_1}\Paren{\Errorrho\log(\ssd+\delta\Errorrho)+\Errorrho}}\phi+\frac{\Heatcapa\Adi\Tconst\ssr^3}{1+\delta\phi}\int_{B_1}\Paren{\frac{\Errorrho^2}{\ssd+\delta\Errorrho}}\\
&+4\pi\ldens\Bracket{\frac{\ssr^5}{9\Paren{1+\delta\phi}^{\frac{11}{3}}}-\frac{\ssr^6}{9\Paren{1+\delta\phi}^{4}}\Paren{\frac{\ssr^3}{1+\delta\phi}+\bV}^{-\frac13}}\dot\phi^2 \\
&+8\pi\ldens\delta\frac{d}{d\delta}\Bracket{\frac{\ssr^5}{9\Paren{1+\delta\phi}^{\frac{11}{3}}}-\frac{\ssr^6}{9\Paren{1+\delta\phi}^{4}}\Paren{\frac{\ssr^3}{1+\delta\phi}+\bV}^{-\frac 13}}\dot\phi^2\\
&+2\pi\ldens\delta^2\frac{d^2}{d^2\delta}\Bracket{\frac{\ssr^5}{9\Paren{1+\delta\phi}^{\frac{11}{3}}}-\frac{\ssr^6}{9\Paren{1+\delta\phi}^{4}}\Paren{\frac{\ssr^3}{1+\delta\phi}+\bV}^{-\frac13}}\dot\phi^2\\
&+\frac{8\pi}{3}\bigg[\frac{5\sigma\ssr^2}{3\Paren{1+\delta\phi}^{\frac83}}+\frac{2\barsigma\ssr^3}{(1+\delta\phi)^3}\Paren{\frac{\ssr^3}{1+\delta \phi}+\bV}^{-\frac 13}\\
&\qquad\quad-\frac{\barsigma\ssr^6}{3(1+\delta\phi)^4}\Paren{\frac{\ssr^3}{1+\delta\phi}+\bV}^{-\frac43}\bigg]\phi^2\triangleq\sum_{i=1}^{10}I_i.
\end{align*}
Therefore, setting $\delta=0$ yields
\begin{align*}
f^{''}(0)
={}&-2{M}{\Heatcapa\Tconst\frac{\Errorrho(1,t)}{\ssd}\phi}+2M\Heatcapa\Tconst\phi^2+M\Heatcapa\Tconst\frac{\Errorrho(1,t)^2}{\ssd^2}+2\Adi M\Heatcapa\Tconst\log\ssd{\phi^2}  \\
&-2\Adi M\Heatcapa\Tconst\log\ssd{\phi^2}-2M\Heatcapa\Adi\Tconst{\phi^2}+\Heatcapa\Adi\Tconst{\ssr^3}{\ssd}^{-1}\int_{B_1}{\Errorrho^2}\\
&+\frac{4\pi\ldens\ssr^5}{9}\Paren{1-\frac{\ssr}{\ssor}}\dot\phi^2+\frac{M}{\ssd}\Bracket{\frac{10\sigma}{3\ssr}+4\frac{\barsigma}{\ssor}-\frac{2}{3}\frac{\barsigma}{\ssor}{\frac{\ssr^3}{\ssor^3}}}\phi^2\\
={}&M\Heatcapa\Tconst\Paren{\frac{\Errorrho(1,t)}{\ssd}-\phi}^2+\frac{4\pi\ldens\ssr^5}{9}\Paren{1-\frac{\ssr}{\ssor}}\dot\phi^2\\
&-M\Paren{\Heatcapa+2\GASCONST}\Tconst{\phi^2}+\Heatcapa\Adi\Tconst{\ssr^3}{\ssd}^{-1}\int_{B_1}{\Errorrho^2}+\frac{M}{\ssd}\Bracket{\frac{10\sigma}{3\ssr}+4\frac{\barsigma}{\ssor}-\frac{2}{3}\frac{\barsigma}{\ssor}{\frac{\ssr^3}{\ssor^3}}}\phi^2\\
\ge{}&M\Heatcapa\Tconst\Paren{\frac{\Errorrho(1,t)}{\ssd}-\frac{3}{4\pi\ssd}\int_{B_{1}}\Errorrho}^2+\frac{4\pi\ldens\ssr^5}{9}\Paren{1-\frac{\ssr}{\ssor}}\dot\phi^2\\
&-\Paren{\Heatcapa+\GASCONST-\vare\GASCONST}\Tconst{\ssr^3}{\ssd}^{-1}\int_{B_{1}}\Errorrho^2+\Paren{\Heatcapa+\GASCONST}\Tconst{\ssr^3}{\ssd}^{-1}\int_{B_1}{\Errorrho^2}\\
&-M(\GASCONST+\vare\GASCONST)\Tconst{\phi^2}+(1+\vare)\frac{M}{\ssd}\Bracket{\frac{2\sigma}{\ssr}+\frac{2\barsigma}{\ssor}}\phi^2\\
&+\frac{M}{\ssd}\Bracket{\frac{4\sigma}{3\ssr}\Paren{1-\frac{3}{2}\vare}+\frac{2\barsigma}{\ssor}\Paren{1-\frac{\ssr^3}{3\ssor^3}-\vare}}\phi^2\\
\ge{}&M\Heatcapa\Tconst\Paren{\frac{\Errorrho(1,t)}{\ssd}-\frac{3}{4\pi\ssd}\int_{B_{1}}\Errorrho}^2+\frac{\ldens\ssr^5}{4\pi\ssd^2}\Paren{1-\frac{\ssr}{\ssor}}\Paren{\int_{B_{1}}\dot\Errorrho}^2+\vare\frac{\GASCONST\Tconst\ssr^3}{\ssd}\int_{B_{1}}\Errorrho^2\\
&+\frac{\ssr^3}{\pi\ssd^2}\Bracket{\frac{\sigma}{\ssr}\Paren{1-\frac{3}{2}\vare}+\frac{\barsigma}{\ssor}\Paren{\frac32-\frac{\ssr^3}{2\ssor^3}-\frac{3}{2}\vare}}\Paren{\int_{B_{1}}\Errorrho}^2,
\end{align*} 
where $\vare\in \Paren{0,2/3}$ and we have used $\Adi=1+\GASCONST/\Heatcapa$,  \eqref{eq-equi-eq}, \eqref{eq-mini-1}, and Cauchy's inequality $\phi^2\le\frac{3}{4\pi\ssd^2}\int_{B_{1}}\Errorrho^2$. Chosen $\vare=1/3$, it holds
\begin{align*}
\frac{1}{2}f^{''}(0)\delta^2\ge{}& \frac{M\Heatcapa\Tconst}{2}\Paren{\frac{\Errorrho(1,t)}{\ssd}-\frac{3}{4\pi\ssd}\int_{B_{1}}\Errorrho}^2\delta^2+\frac{\ldens\ssr^5}{8\pi\ssd^2}\Paren{1-\frac{\ssr}{\ssor}}\Paren{\int_{B_{1}}\dot\Errorrho}^2\delta^2\\
&+\frac{\GASCONST\Tconst{\ssr^3}}{6\ssd}\int_{B_{1}}\Errorrho^2\delta^2+\frac{\ssr^3}{2\pi\ssd^2}\Bracket{\frac{\sigma}{2\ssr}+\frac{\barsigma}{\ssor}\Paren{1-\frac{\ssr^3}{2\ssor^3}}}\Paren{\int_{B_{1}}\Errorrho}^2\delta^2\\
\triangleq&J_1\delta^2+J_2\delta^2+J_3\delta^2+J_4\delta^2. 
\end{align*}

To deal with $f^{'''}$, using $\delta\le1$ and $\norm{\Errorrho(\cdot,t)}_{L^\infty(B_1)}\le\ssd/2$, we note that all the denominators of $I_i$ are bounded, i.e.,  $1/2\le1+\xi(\delta)\phi\le3/2$, $\ssd/2\le\ssd+\xi(\delta)\Errorrho(r,t)\le3\ssd/2$. 
With this lower bound, all the terms in  $f^{'''}(\xi(\delta))\delta^3/6$ coming from $I_1$, $I_2$, $I_4$, $I_5$, $I_6$ and $I_{10}$ can be absorbed into $J_3\delta^2$ and $J_4\delta^2\ge \frac{\ssr^3}{4\pi\ssd^2}\Paren{\frac{\sigma}{\ssr}+\frac{\barsigma}{\ssor}}\Paren{\int_{B_{1}}\Errorrho}^2\delta^2$ for $\delta\le\delta_0^1/\Paren{1+\module{\log\ssd}}$, where $\delta_0^1>0$ depends only on $\Heatcapa/\GASCONST$. We show the calculation of $I_1,I_5,I_{10}$ and the others are similar or easier. The derivatives of the first two terms can be absorbed by using \eqref{eq-equi-eq}
\begin{align*}
\module{\frac{dI_1}{d\delta}(\xi(\delta))\delta^3}
={}&\module{\frac{\GASCONST\Tconst\ssd}{\GASCONST}\frac{\Heatcapa\Errorrho(1,t)}{(1+\xi(\delta)\phi)^3\ssd}\frac{3\ssr^3}{\pi\ssd^2}\Paren{\int_{B_{1}}\Errorrho}^2}\delta^3
\le \Paren{\frac{72\Heatcapa}{\GASCONST}\delta}J_4\delta^2,\\
\module{\frac{dI_5}{d\delta}(\xi(\delta))\delta^3}={}&\left|-2\Paren{\frac{2\Heatcapa\Adi\Tconst\ssr^3}{\Paren{1+\xi(\delta)\phi}^3}\int_{B_1}\Paren{\Errorrho\log(\ssd+\xi(\delta)\Errorrho)+\Errorrho}}\phi^2\right.\\
&\left.+\Paren{\frac{2\Heatcapa\Adi\Tconst\ssr^3}{\Paren{1+\xi(\delta)\phi}^2}\int_{B_1}\Paren{\frac{\Errorrho^2}{\ssd+\xi(\delta)\Errorrho}}}\phi\right|\delta^3\\
\le{}&\frac{12\Heatcapa\Adi}{\GASCONST}\GASCONST\Tconst\ssd\frac{\ssr^3}{\pi\ssd^2}\max\Brace{\module{\log\frac{3\ssd}{2}},\module{\log\frac{\ssd}{2}}}\Paren{\int_{B_{1}}\Errorrho}^2\delta^3\\
&+\frac{12\Heatcapa\Adi}{\GASCONST}\GASCONST\Tconst\ssd\frac{\ssr^3}{\pi\ssd^2}\Paren{\int_{B_{1}}\Errorrho}^2\delta^3+\frac{12\Heatcapa\Adi}{\GASCONST}\frac{\GASCONST\Tconst \ssr^3}{3\ssd}\int_{B_1}\Errorrho^2\delta^3\\
\le{}&\Bracket{\frac{200\Paren{\Heatcapa+\GASCONST}}{\GASCONST}\Paren{1+\module{\log\ssd}}\delta}J_4\delta^2+\Paren{\frac{24\Paren{\Heatcapa+\GASCONST}}{\GASCONST}\delta}J_4\delta^2.
\end{align*} 
For the derivative of $I_{10}$, it is sufficient to notice that $|\phi|\le 1/2$ from \eqref{eq-mini-1}, and   
\begin{align*}
\module{\frac{d}{d\delta}\Bracket{\Paren{\frac{\ssr^3}{1+\delta\phi}+\bV}^{-\frac13}}(\xi(\delta))}&=\module{\frac13\Paren{\frac{\ssr^3}{1+\xi(\delta)\phi}+\bV}^{-\frac43}\frac{\ssr^3\phi}{\Paren{1+\xi(\delta) \phi}^2}}\le\frac{100}{\ssor},\\
\module{\frac{d}{d\delta}\Bracket{\Paren{\frac{\ssr^3}{1+\delta\phi}+\bV}^{-\frac43}}(\xi(\delta))}&=\module{\frac43\Paren{\frac{\ssr^3}{1+\xi(\delta)\phi}+\bV}^{-\frac73}\frac{\ssr^3\phi}{\Paren{1+\xi(\delta) \phi}^2}}\le\frac{100}{\ssor^4},
\end{align*} 
where we have used 
\begin{equation*}
\Paren{\frac{\ssr^3}{1+\xi(\delta)\phi}+\bV}^{-\frac13}\le\Paren{\frac{2}{3}\ssr^3+\frac{2}{3}\bV}^{-\frac13}\le\Paren{\frac{3}{2}}^{\frac13}\ssor^{-1},
\end{equation*}
which is independent of the liquid volume $V\in (0,\infty)$. Then, $\frac{d}{d\delta}I_{10}(\xi(\delta))\delta^3/6$ can be absorbed into $J_4\delta^2$. 

For the derivative of $I_3$ which contains the term $\Errorrho(1,t)^3$, we have 
\begin{equation*}
\module{\frac{dI_3}{d\delta}(\xi(\delta))\delta^3}=\module{\frac{2M\Heatcapa\Tconst\Errorrho(1,t)^3}{\Paren{\ssd+\xi(\delta)\Errorrho(1,t)}^3}}\delta^3\le\Paren{32\delta}J_1\delta^2+\Paren{\frac{16\Heatcapa}{\GASCONST}\delta}J_3\delta^2,
\end{equation*}
where we have used
\begin{align*}
|\Errorrho(1,t)||\Errorrho(1,t)|^2
&\le\frac{\ssd}{2}\Bracket{\module{\frac{\Errorrho(1,t)}{\ssd}-\frac1{|B_1|\ssd}\int_{B_1}\Errorrho}+|B_1|^{-\frac12}\ssd^{-1}\Paren{\int_{B_1}\Errorrho^2}^{\frac12}}^2\ssd^2\\
&\le\ssd^3\Bracket{\Paren{\frac{\Errorrho(1,t)}{\ssd}-\frac1{|B_1|\ssd}\int_{B_1}\Errorrho}^2+|B_1|^{-1}\ssd^{-2}\Paren{\int_{B_1}\Errorrho^2}}.
\end{align*}
Therefore, this term can also be absorbed into $f^{''}(0)\delta^2/2$ for all $0<\delta<\delta_0^2$, where $\delta_0^2>0$ depends on $\Heatcapa/\GASCONST$. 

To deal with the derivatives of $I_7,I_8$ and $I_9$ containing $(\int_{B_{1}}\dot\Errorrho)^2$, we need to calculate the first three order derivative of $\ldens \ssr^5\psi(\delta)\dot\phi^2$ where
\begin{equation*}
\psi(\delta)=\Paren{1+\delta\phi}^{-\frac{11}{3}}-\ssr\Paren{1+\delta\phi}^{-4}\Paren{\ssr^3\Paren{1+\delta\phi}^{-1}+\bV}^{-\frac 13}.
\end{equation*} 
We only control $d\psi/d\delta $ since $d^2\psi/d^2\delta,d^3\psi/d^3\delta$ can be calculated similarly. By direct calculation, we obtain
\begin{equation*}
\frac{d\psi}{d\delta}(\delta)=\frac{\phi}{3(1+\delta\phi)^{\frac{14}{3}}}\Bracket{12\Paren{1+\bV\ssr^{-3}\Paren{1+\delta\phi}}^{-\frac13}-\Paren{1+\bV\ssr^{-3}\Paren{1+\delta\phi}}^{-\frac43}-11}.
\end{equation*}
Therefore, from $|\phi|\le 1/2$ and $\frac{\ssor-\ssr}{\ssor}=1-\Paren{1+\bV\ssr^{-3}}^{-\frac{1}{3}}$, it follows that
\begin{align*}
\module{\frac{d\psi}{d\delta}(\xi(\delta))}
\le{}&6\Bracket{12\module{1-\Paren{1+\bV\ssr^{-3}\Paren{1+1/2}}^{-\frac13}}+\module{1-\Paren{1+\bV\ssr^{-3}\Paren{1+1/2}}^{-\frac43}}}\\
\le{}&6\Bracket{19\module{1-\Paren{1+\bV\ssr^{-3}}^{-\frac13}}+7\module{1-\Paren{1+\bV\ssr^{-3}}^{-\frac13}}}\\
\le{}&200\Paren{\frac{\ssor-\ssr}{\ssor}},
\end{align*}
and $\bV\ssr^{-3}\le \delta_0^3$
where $\delta_0^3>0$ is a small constant independent of $\Tconst,\ldens,\ssd,\ssr,M$ and $V$. Moreover, for all $\alpha>0$, we have 
\begin{equation*}
1-\Paren{1+\bV\ssr^{-3}\Paren{1+1/2}}^{-\frac\alpha3}\le1\le\frac{1}{1-\Paren{1+\delta_0^3}^{-\frac{1}{3}}}\frac{\ssor-\ssr}{\ssor},\quad\bV\ssr^{-3}\ge\delta_0^3.
\end{equation*}
We conclude that
\begin{equation*}
\module{\frac{d\psi}{d\delta}(\xi(\delta))}\le100\Bracket{\Paren{1-\Paren{1+\delta_0^3}^{-\frac{1}{3}}}^{-1}+1}\frac{\ssor-\ssr}{\ssor},
\end{equation*} 
for all $V$ and $\ssr$.  
Therefore, all the terms $\frac{d}{d\delta}(I_7+I_8+I_9)(\xi(\delta))\delta^3$  containing  $\dot\phi^2$ can be absorbed into $J_2\delta^2$, for all $\delta<\delta_0^4$ with $\delta_0^4$ independent of $\Tconst,\ldens,\ssd,\ssr,M$ and $V$. 

Summarizing the above results and recalling that we choose any $(\gdens_\delta(\cdot,t),R_\delta(t))$ satisfying \eqref{eq-mini-0}. We conclude that for all $(\gdens(\cdot,t),R(t))$ satisfying 
\begin{equation*}
\norm{\gdens(\cdot,t)-\ssd}_{L^\infty(B_{R(t)})}\le{\delta_0\ssd}/\Paren{1+\module{\log\ssd}},
\end{equation*} 
or equivalently, \eqref{eq-minimizer-condition}, 
where $\delta_0>0$ is a constant that depends on $\delta_0^i$ and hence only on $\Heatcapa/\GASCONST$, or $\Adi$, it holds
\begin{align*}
\Energy[\brho,R]\ge{}&\Energy_{\dagger}+\frac{1}{4}\left\lbrace M\Heatcapa\Tconst\Paren{\frac{\brho(1,t)-\ssd}{\ssd}-\frac{3}{4\pi\ssd}\int_{B_{1}}\Paren{\brho(y,t)-\ssd}}^2\right.\\
&\qquad\quad\ \ +\frac{\ldens\ssr^5}{4\pi\ssd^2}\Paren{1-\frac{\ssr}{\ssor}}\Paren{\int_{B_{1}}\dot\brho}^2+\frac{\GASCONST\Tconst{\ssr^3}}{3\ssd}\int_{B_{1}}\Paren{\brho(y,t)-\ssd}^2\\
&\qquad\quad\ \ +\left.\frac{\ssr^3}{\pi\ssd^2}\Bracket{\frac{\sigma}{2\ssr}+\frac{\barsigma}{\ssor}\Paren{1-\frac{\ssr^3}{2\ssor^3}}}\Paren{\int_{B_{1}}\Paren{\brho(y,t)-\ssd}}^2\right\rbrace,
\end{align*}
where we have used  $\brho(\cdot,t)=\brho_{\delta}(\cdot,t)=\ssd+\delta\Errorrho(\cdot,t)$. This completes the proof.
\end{proof}  
\section{Nonlinear and exponential asymptotic stability} \label{s:proof} 

\subsection{Proof of the nonlinear asymptotic stability} 

In this subsection, we prove the first part of main Theorem \ref{t-main result}, especially the asymptotic stability in \eqref{eq-main- results-i}. 

Given the equilibrium state $(\ssd[M,V],\ssr[M,V])$ determined by the mass-volume pair $(M,V)$ and any initial data $(\gdens_0,R_0,\ptR_0)$, since free boundary problem \eqref{eq-l-main}--\eqref{eq-b-main} is equivalent to system \eqref{eq-main-syst} in the spherical case, the energy dissipation \eqref{eq-energy-dissipation} reduces to
\begin{equation}\label{eq-energy-dissipation-sph}
\dot\Energy(\tau)=-\Thermal\Tconst\int_{B_{R(\tau)}}\frac{|\nabla_r\gdens(|x|,\tau)|^2}{\gdens(|x|,\tau)^2}dx-16\pi\mu_l\frac{\bV R(\tau)}{R(\tau)^3+\bV}\ptR(\tau)^2,
\end{equation} 
where $\nabla_r$ denotes the radial gradient. In the above, we have used the velocity formula \eqref{eq-pro2-3}, $\Omn(t)=B_{R(t)}$ and $\Omp(t)=B_{\bR(t)}\setminus\overline{B_{R(t)}}$. 

Integrating from $0$ to $t$ and applying the local minimizer \eqref{eq-minimizer} together with the Lyapunov stability \eqref{eq-2000-1}--\eqref{eq-2000-3}, we obtain
\begin{align*}
&\Thermal\Tconst\int_0^t\int_{B_{R(\tau)}}\frac{|\nabla_r\gdens(|x|,\tau)|^2}{\gdens(|x|,\tau)^2}dxd\tau+16\pi\mu_l\int_0^t\frac{\bV R(\tau)\ptR(\tau)^2}{R(\tau)^3+\bV}d\tau\\
\le{}&\Energy_0-\Energy_{\dagger}-\Paren{\Energy(t)-\Energy_{\dagger}}\\
\le{}&\Energy_0-\Energy_{\dagger},
\end{align*}
provided the constant $\eta_0>0$ is small enough, which is independent of the initial data. Also, we have denoted $\Energy_0=\Energy[\gdens_0,R_0]$ and $\Energy_{\dagger}=\Energy[\ssd,\ssr]$ in the above. By the regularity of $(\gdens(x,t),R(t))$ in \eqref{eq-2000-3}, we deduce that the time-dependent functions
\begin{equation*}
\int_{B_{R(\tau)}}\frac{|\nabla_r\gdens(|x|,\tau)|^2}{\gdens(|x|,\tau)^2}dx,\ \text{and}\ R(\tau)\ptR(\tau)^2\ge 0
\end{equation*} 
are uniformly continuous. This, combined with the bound $0\le\bV/\Paren{R(\tau)^3+\bV}\le1$, follows that the function  $\bV R(\tau)(\ptR(\tau))^2/\Paren{R(\tau)^3+\bV}$ 
is also uniformly continuous independent of the liquid volume $V$. 

The remaining proof of \eqref{eq-main- results-i} is similar to that of  \cite[Proposition 8.1]{Lai2023}, as we can apply Theorem \ref{t-equilibria}, Barbalat's lemmas in stability theory, and interpolations. This proof is valid because we have  shown that  the algebraic system \eqref{eq-equi-eq}, or equivalently, the equation
\begin{equation*}
(\surftenratio^{3}+1)R^{9}-3IR^{7}+\bV R^{6}+3I^{2}R^{5}-3I\bV R^{4}-I^{3}R^{3}+3I^{2}\bV R^{2}-I^{3}\bV=0,
\end{equation*}
has a unique positive root for any fixed liquid volume.  
\subsection{Proof of the exponential convergence rate}\label{s-exponential convergence}  

The exponential convergence in  Theorem \ref{t-main result} is established through the utilization of the center manifold theory in Appendix \ref{appendix-4}, where the relevant definitions are provided.

\vspace{6pt} 
\noindent \textbf{Step 1:} We start with transforming free boundary problem \eqref{eq-main-syst} into an equivalent system within an appropriate Banach space.  Specifically, we select $Z=\ell^2$ as specified in \eqref{eq-app-4-1}.

\begin{proposition}\label{p-infinite-ODE}
Under the assumptions of Theorem \ref{t-main result}, if we abbreviate the equilibrium state $(\ssd[M,V],\ssr[M,V])=(\ssd,\ssr)$, and decompose the global-in-time solution $(\gdens,R)$ as follows
\begin{equation}\label{eq-change-var}
\begin{cases}
\gdens(R(t)y,t)=\ssd+\Errorrho(y,t)=\ssd+\Errorrho_{1}(y,t)+\Errorrho_{2}(t),\ 0\le y\le1,&\ t>0,\\ 
\Errorrho_{2}(t)=\gdens(R(t),t)-\ssd,&\ t>0,\\
R(t)=\ssr+\ErrorR(t),&\ t>0,
\end{cases} 
\end{equation}  
then, system \eqref{eq-main-syst} is converted to the following initial-boundary value problem 
\begin{subnumcases}{\label{eq-IBVP}} 
\parl_t\Errorrho_{1}(y,t)=\bkappa\Laplace\Errorrho_{1}(y,t)-\Paren{1-\frac{1}{\Adi}}\dot\Errorrho_{2}+\Pi,\ 0\le y\le1,\ \Errorrho_{1}(1,\cdot)\equiv0, &$t>0$,\label{eq-uRz-u}\\
\ptErrorR=-\frac{\ssr}{\ssd}\Paren{\bkappa\parl_y\Errorrho_{1}(1,t)+\frac{1}{3\Adi}\dot\Errorrho_{2}}+\Phi,&$t>0$,\label{eq-uRz-R}\\
\Errorrho_{2}=\frac1{\GASCONST\Tconst}\Bracket{-\Paren{\frac{2\sigma}{\ssr^2}+\frac{2\barsigma\ssr^2}{\ssor^{4}}}\ErrorR+\frac{4\mu_l\bV}{\ssr(\ssr^3+\bV)}\dot\ErrorR+\ldens\tR \ddot\ErrorR}+\Psi,&$t>0$,\label{eq-uRz-z}
\end{subnumcases} 
with the initial condition
\begin{equation}\label{eq-uRZ-initial}
(\Errorrho_{1}(y,0),\Errorrho_{2}(0),\ErrorR(0),\dot\ErrorR(0))=(\gdens_0(R_0y)-\gdens_0(R_0),\gdens_0(R_0)-\ssd,R_0-\ssr,\ptR_0),
\end{equation}
for $0\le y\le 1$. Moreover, in system \eqref{eq-IBVP}, we denote $\tR=\ssr-{\ssr^2}/{\ssor}$, $\bkappa=\Thermal/(\ssr^2\ssd\Adi\Heatcapa)$, and the following nonlinear terms 
\begin{subequations}
\label{eq-F-G-H}
\begin{align}
\Pi={}&\frac{\Thermal}{\Adi\Heatcapa}\Bracket{\frac1{(\ssr+\ErrorR)^2(\ssd+\Errorrho(y,t))}-\frac1{\ssr^2\ssd}}\Laplace_y\Errorrho_{1}(y,t)\nonumber\\ 
&-\frac{\Thermal}{\Adi \Heatcapa}  \frac{|\nabla_y\Errorrho_{1}(y,t)|^2}{(\ssr+\ErrorR)^2(\ssd+\Errorrho(y,t))^2}+\frac1\Adi\frac{\dot\Errorrho_{2}}{\ssd+\Errorrho_{2}}\Paren{\frac13y\parl_y\Errorrho_{1}(y,t)+\Errorrho_{1}(y,t)},\label{eq-FGH-F}\\
\Phi={}&-\frac{\Thermal}{\Adi\Heatcapa}\Bracket{\frac1{(\ssr+\ErrorR)(\ssd+\Errorrho_{2})^2}-\frac1{\ssr\ssd^2}}\parl_y\Errorrho_{1}(1,t)-\frac{\ErrorR\dot\Errorrho_{2}}{3\Adi(\ssd+\Errorrho_{2})}\nonumber\\
&+\frac{\ssr}{3\Adi}\frac{\Errorrho_{2}\dot\Errorrho_{2}}{\ssd(\ssd+\Errorrho_{2})},\label{eq-FGH-G}\\
\Psi={}&\frac1{\GASCONST\Tconst}\left[4\mu_l\bV\Paren{\frac{1}{R(R^3+\bV)}-\frac{1}{\ssr(\ssr^3+\bV)}}\dot\ErrorR+\Paren{\frac{2\sigma}{\ssr^2}\ErrorR+\frac{2\sigma}{\ssr+\ErrorR}-\frac{2\sigma}{\ssr}}\right.\nonumber\\
&\qquad\quad\left.+\Paren{\frac{2\barsigma\ssr^2}{\ssor^{4}}\ErrorR+\frac{2\barsigma}{\bR}-\frac{2\barsigma}{\ssor}}\right]\nonumber\\
&+\frac{\ldens}{\GASCONST\Tconst}\Bracket{\ErrorR\ddot\ErrorR+\Paren{\frac{\ssr^2}{\ssor}\ddot\ErrorR-\frac{R^2}{\bR}\ddot\ErrorR}+\Paren{\frac32-\frac{2R}{\bR}+\frac{R^4}{2\bR^4}}\ptR^2}.\label{eq-FGH-H}
\end{align}
\end{subequations} 

We denote the normalized radial Dirichlet eigenfunctions defined on the unit ball $B_1$ by $\Brace{\Xi_j(y)=\frac{\sin(j\pi y)}{\sqrt{2\pi}y}}_{j=1}^{\infty}$, which satisfies $-\Laplace_y\Xi_j=\eigv_j\Xi_j$ with $\eigv_j=(j\pi)^2,\Xi_j|_{y=1}=0$, and $\int_{B_1}\Xi_j^2(|x|)dx=1$  for $j\ge 1$. We expand $\Errorrho_{1}$ as 
\begin{equation}\label{eq-u-expansion}
\Errorrho_{1}(y,t)=\sum_{j=1}^\infty \theta_j(t)\Xi_j(y).
\end{equation} 

Then, system \eqref{eq-IBVP} is further equivalent to the following infinite dimensional dynamical system for  $\zz=\begin{bmatrix}\Errorrho_{2}&\ErrorR&\ptErrorR&\theta_1&\theta_2&\cdots\end{bmatrix}^\top$. That is,
\begin{equation}\label{eq-Lw+Nw}
\dot\zz=\Linearpart\zz+\Nonlpart(\zz,\dot\zz)=\Linearpart\zz+\Nonlpart^1(\zz)\dot\zz+\Nonlpart^0(\zz),
\end{equation}  
where both the linear operator $\Linearpart$ and the nonlinear term $\Nonlpart(\zz,\dot\zz)$ are defined in \eqref{eq-N(w,}. The terms $\Nonlpart^1$ and $\Nonlpart^0$ are given in \eqref{eq-def-N1N0}. Finally, the initial condition of problem \eqref{eq-Lw+Nw} can be deduced from the original condition \eqref{eq-uRZ-initial}.
\end{proposition} 
\begin{proof}
To begin with, \eqref{eq-uRz-u} and \eqref{eq-uRz-R} are derived by substituting \eqref{eq-change-var} into  \eqref{eq-main-syst}.
We point out that \eqref{eq-uRz-z} is deduced by applying \eqref{eq-equi-eq}, Taylor's theorem $\Paren{\ssr+x}^{-1}=\ssr^{-1}-\ssr^{-2}x+\cdots$ and $((\ssr+x)^3+\bV)^{-\frac 13}=(\ssr^3+\bV)^{-\frac 13}-\ssr^2\ssor^{-4}x+\cdots$. 
More precisely, we calculate as follows
\begin{align*}
\Errorrho_{2}(t)={}&\frac1{\GASCONST \Tconst}\left[4\mu_l\frac{\bV}{R(R^3+\bV)}\dot\ErrorR+\frac{2\sigma}{\ssr+\ErrorR}-\frac{2\sigma}{\ssr}+\frac{2\barsigma}{\bR}-\frac{2\barsigma}{\ssor}\right.\\
&\quad\quad\ \left.+\ldens\Paren{R\pttR-\frac{R^2}{\bR}\pttR+\Paren{\frac32-\frac{2R}{\bR}+\frac{R^4}{2\bR^4}}\ptR^2}\right]\\
={}&\frac1{\GASCONST\Tconst}\Bracket{4\mu_l\frac{\bV}{\ssr(\ssr^3+\bV)}\dot\ErrorR-\frac{2\sigma}{\ssr^2}\ErrorR-\frac{2\barsigma\ssr^2}{\ssor^{4}}\ErrorR+\ldens\tR\ddot\ErrorR}\\
&+\frac1{\GASCONST\Tconst}\left[\Paren{4\mu_l\frac{\bV}{R(R^3+\bV)}\dot\ErrorR-4\mu_l\frac{\bV}{\ssr(\ssr^3+\bV)}\dot\ErrorR}\right.\\
&\qquad\qquad\left.+\Paren{\frac{2\sigma}{\ssr^2}\ErrorR+\frac{2\sigma}{\ssr+\ErrorR}-\frac{2\sigma}{\ssr}}+\Paren{\frac{2\barsigma\ssr^2}{\ssor^{4}}\ErrorR+\frac{2\barsigma}{\bR}-\frac{2\barsigma}{\ssor}}\right] \\
&+\frac{\ldens}{\GASCONST\Tconst}\Bracket{\ErrorR\ddot\ErrorR+\Paren{\frac{\ssr^2}{\ssor}\ddot\ErrorR-\frac{R^2}{\bR}\ddot\ErrorR}+\Paren{\frac32-\frac{2R}{\bR}+\frac{R^4}{2\bR^4}}\ptR^2},
\end{align*} 
where we have denoted $\tR=\ssr-{\ssr^2}/{\ssor}$. Then, equations \eqref{eq-uRz-z} and \eqref{eq-FGH-H} follow. 	

Next, substituting \eqref{eq-u-expansion} into \eqref{eq-uRz-u} and computing the inner product in $L^2(B_1)$ with $\Xi_k(y)$ yields the following
\begin{equation}\label{eq-dot-ck}
\dot\theta_k=-\bkappa\eigv_k\theta_k-\coeff_k\dot\Errorrho_{2}+\Pi_k,\quad\coeff_k=\frac{(-1)^{k-1}2^{3/2}(\Adi-1)}{\sqrt\pi\Adi k},\quad\Pi_k=\int_{B_1}\Pi\Xi_kdy.
\end{equation} 
Using $\parl_y\Xi_j(1)=\sqrt{\pi}(-1)^jj/\sqrt{2}$, the second equation \eqref{eq-uRz-R} becomes
\begin{equation}\label{eq-dot-R}
\ptErrorR 
=\sum_{j=1}^\infty\theta_j\omega_j-\frac{\ssr}{3\Adi\ssd}\dot\Errorrho_{2}+\Phi,\quad\omega_j=-\frac{\ssr\bkappa}{\ssd}\sqrt{\frac{\pi}2}(-1)^jj.
\end{equation} 
Moreover, the third equation \eqref{eq-uRz-z} implies
\begin{equation*}
\frac{\GASCONST\Tconst}{\ldens\tR}\Errorrho_{2}=-\Paren{\frac{2\sigma}{\ldens\tR\ssr^2}+\frac{2\barsigma\ssr^2}{\ldens\tR\ssor^{4}}}\ErrorR+\frac{4\mu_l\bV}{\ldens\tR\ssr(\ssr^3+\bV)}\dot\ErrorR+\ddot\ErrorR+\frac{\GASCONST\Tconst\Psi}{\ldens\tR}.
\end{equation*} 
As a consequence, problem \eqref{eq-IBVP} forms an infinite-dimensional dynamical system 
\begin{align*}
&\begin{bmatrix}
\frac{\ssr}{3\Adi\ssd}&1&0&0&0&\cdots\\
0&1&0&0&0&\cdots\\
0&0&1&0&0&\cdots\\
\coeff_1&0&0&1&0&\cdots\\
\coeff_2&0&0&0&1&\cdots\\
\vdots&\vdots&\vdots&\vdots&\vdots&\ddots
\end{bmatrix}
\begin{bmatrix}
\Errorrho_{2}\\ \ErrorR\\ \ptErrorR\\ \theta_1\\ \theta_2\\ \vdots
\end{bmatrix}^\prime=\\
&
\begin{bmatrix}
0&0&0&\omega_1&\omega_2&\cdots\\
0&0&1&0&0&\cdots\\
\frac{\GASCONST \Tconst}{\ldens\tR}& \frac{2\sigma}{\ldens\tR \ssr^2}+\frac{2\barsigma \ssr^2 }{\ldens\tR \ssor^{4}}& \frac{-4\mu_l\bV }{\ldens\tR \ssr(\ssr^3+\bV)}&0&0&\cdots\\
0&0&0&-\bkappa\eigv_1&0&\cdots\\
0&0&0&0&-\bkappa\eigv_2&\cdots\\
\vdots&\vdots&\vdots&\vdots&\vdots&\ddots
\end{bmatrix}
\begin{bmatrix}
\Errorrho_{2}\\ \ErrorR\\ \ptErrorR\\ \theta_1\\ \theta_2\\ \vdots
\end{bmatrix}
+ 
\begin{bmatrix}
\Phi\\ 0\\ -\frac{\GASCONST \Tconst}{\ldens\tR} \Psi\\ \Pi_1\\ \Pi_2\\ \vdots
\end{bmatrix}.
\end{align*}
Multiplying both sides of the above equation by the inverse of the infinite-dimensional matrix and denoting $\zz=\begin{bmatrix}
\Errorrho_{2}&\ErrorR&\ptErrorR& \theta_1&\theta_2&\cdots
\end{bmatrix}^\top$,  \eqref{eq-Lw+Nw} follows, where the linear operator $\Linearpart$ equals  
\begin{equation}\label{eq-matrix-L}
\begin{bmatrix}
0&\!\!0&\!\!-\frac{3\Adi\ssd}{\ssr}&\!\!\omega_1\frac{3\Adi\ssd}{\ssr}&\!\!\omega_2\frac{3\Adi\ssd}{\ssr}&\!\!\cdots\\
0&\!\!0&\!\!1&\!\!0&\!\!0&\!\!\cdots\\
\frac{\GASCONST\Tconst}{\ldens\tR}&\!\!\frac{2\sigma\ssor^{4}+2\barsigma\ssr^4}{\ldens\tR\ssr^2\ssor^{4}}&\!\! \frac{-4\mu_l\bV }{\ldens\tR \ssr(\ssr^3+\bV)}&\!\!0&\!\!0&\!\!\cdots\\
0&\!\!0&\!\!\coeff_1\frac{3\Adi\ssd}{\ssr}&\!\!-\coeff_1\omega_1\frac{3\Adi\ssd}{\ssr}-\bkappa\eigv_1&\!\!-\coeff_1\omega_2\frac{3\Adi\ssd}{\ssr}&\!\!\cdots\\
0&\!\!0&\!\!\coeff_2\frac{3\Adi\ssd}{\ssr}&\!\!-\coeff_2\omega_1\frac{3\Adi\ssd}{\ssr}&\!\!-\coeff_2\omega_2\frac{3\Adi\ssd}{\ssr}-\bkappa\eigv_2&\!\!\cdots\\
\vdots&\!\!\vdots&\!\!\vdots&\!\!\vdots&\!\!\vdots&\!\!\ddots
\end{bmatrix},
\end{equation} 
and the nonlinear term $\Nonlpart=\Nonlpart(\zz,\dot\zz)$ is equal to
\begin{equation}\label{eq-N(w,}
\begin{bmatrix}
\frac{3\Adi\ssd}{\ssr}\Phi&0&-\frac{\GASCONST\Tconst}{\ldens\tR}\Psi&-\coeff_1\frac{3\Adi\ssd}{\ssr}\Phi+\Pi_1&-\coeff_2\frac{3\Adi\ssd}{\ssr}\Phi+\Pi_2&\cdots
\end{bmatrix}^\top,
\end{equation}
where $\Phi=\Phi(\zz,\dot\zz),\Psi=\Psi(\zz,\dot\zz),\Pi_k=\Pi_k(\zz,\dot\zz),\coeff_k$ and $\omega_k$ are defined in \eqref{eq-FGH-G}, \eqref{eq-FGH-H}, \eqref{eq-dot-ck} and \eqref{eq-dot-R}, respectively. Writing $\Nonlpart(\zz,\dot\zz)=\Nonlpart(\zz,\ww)$ (denoting $\ww=\dot\zz$), we further decompose
\begin{equation}\label{eq-def-nonl-term}
\begin{cases}
\Pi(\zz,\ww)=\Anglebracket{\bPi^1(\zz),\ww}+\Pi^0(\zz),\\ \Phi(\zz,\ww)=\Anglebracket{\bPhi^1(\zz),\ww}+\Phi^0(\zz),\\ \Psi(\zz,\ww)=\Anglebracket{\bPsi^1(\zz),\ww}+\Psi^0(\zz), 
\end{cases}
\end{equation} 
i.e.,  
\begin{align*}
\Pi(\zz,\ww)={}&\Anglebracket{\begin{bmatrix}\dfrac{a}{\Adi(\ssd+\Errorrho_{2})}\sum_{j=1}^\infty\theta_j\Paren{\dfrac{y}{3}\parl_y\Xi_j+\Xi_j}&0&0&\cdots\end{bmatrix}^\top,\ww} \\
&+\frac{\module{\nabla_y\Errorrho_{1}}^2}{(\ssr+\ErrorR)^2\Paren{\ssd+\Errorrho}^2}-\frac{\Thermal}{\Adi\Heatcapa}\Bracket{\frac1{(\ssr+\ErrorR)^2\Paren{\ssd+\Errorrho}}-\frac1{\ssr^2\ssd}}\sum_{j=1}^\infty\eigv_j\theta_j\Xi_j,\\ 
\Phi(\zz,\ww)={}&\Anglebracket{\begin{bmatrix}-\dfrac{\ErrorR}{3\Adi(\ssd+\Errorrho_{2})}+\dfrac{\ssr\Errorrho_{2}}{3\Adi\ssd(\ssd+\Errorrho_{2})}&0&0&\cdots\end{bmatrix}^\top,\ww}\\
&-\frac{\Thermal}{\Adi\Heatcapa}\Bracket{\frac1{(\ssr+\ErrorR)(\ssd+\Errorrho_{2})^2}-\frac1{\ssr\ssd^2}}\sum_{j=1}^\infty\sqrt{\frac\pi2}(-1)^j j\theta_j,\\
\Psi(\zz,\ww)={}&\Anglebracket{\begin{bmatrix}0&0&\dfrac{\ldens}{\GASCONST\Tconst}\Bracket{\ErrorR+\Paren{\dfrac{\ssr^2}{\ssor}-\dfrac{R^2}{\bR}}}&0&0&\cdots\end{bmatrix}^\top,\ww}\\
&+\frac1{\GASCONST\Tconst}\left[4\mu_l\bV\Paren{\frac{1}{R(R^3+\bV)}-\frac{1}{\ssr(\ssr^3+\bV)}}\dot\ErrorR+\Paren{\frac{2\sigma}{\ssr^2}\ErrorR+\frac{2\sigma}{\ssr+\ErrorR}-\frac{2\sigma}{\ssr}}\right.\\
&\qquad\quad\ \  \left.+\Paren{\frac{2\barsigma\ssr^2}{\ssor^{4}}\ErrorR+\frac{2\barsigma}{\bR}-\frac{2\barsigma}{\ssor}}+\ldens\Paren{\frac32-\frac{2R}{\bR}+\frac{R^4}{2\bR^4}}\ptR^2\right].
\end{align*}  
Above, $\Anglebracket{\cdot,\cdot}$ denotes inner product in the Hilbert space $\ell^2$. Therefore, the nonlinear term $\Nonlpart(\zz,\ww)$ can be further decomposed as follows
\begin{equation}\label{eq-N1P+N0}
\Nonlpart(\zz,\ww)=\Nonlpart^1(\zz)\ww+\Nonlpart^0(\zz),
\end{equation} 
where
\begin{equation}\label{eq-def-N1N0}
\begin{cases}
\Nonlpart^1= 
\begin{bmatrix}
\frac{3\Adi\ssd}{\ssr}\bPhi^1&\Zero&
-\frac{\GASCONST\Tconst}{\ldens\tR}\bPsi^1&\bPi_1^1-\coeff_1\frac{3\Adi\ssd}{\ssr}\bPhi^1&\bPi_2^1-\coeff_2\frac{3\Adi\ssd}{\ssr}\bPhi^1&\cdots
\end{bmatrix}^\top\!,\\
\Nonlpart^0=
\begin{bmatrix}
\frac{3\Adi\ssd}{\ssr}\Phi^0&0&-\frac{\GASCONST\Tconst}{\ldens\tR}\Psi^0&\Pi_1^0-\coeff_1\frac{3\Adi\ssd}{\ssr}\Phi^0&\Pi_2^0-\coeff_2\frac{3\Adi\ssd}{\ssr}\Phi^0&\cdots
\end{bmatrix}^\top\!.
\end{cases}
\end{equation} 
Here, $\Nonlpart^1(\zz)\ww$ denotes the point-wise inner product $\begin{bmatrix}
\Anglebracket{\frac{3\Adi\ssd}{\ssr}\bPhi^1(\zz),\ww}&\Anglebracket{\Zero,\ww}&\cdots
\end{bmatrix}^\top$,
$\Pi_j^0(\zz)=\int_{B_1}\Pi^0(\zz) \Xi_jdx$, and
$\bPi_j^1(\zz)=\int_{B_1}\bPi^1(\zz) \Xi_jdx$. This completes the proof of the proposition.
\end{proof}  

\noindent \textbf{Step 2:} The asymptotic stability result (part (i)) in Theorem \ref{t-main result} can now be expressed for system \eqref{eq-Lw+Nw} within the Banach space $\ell^2$ (cf. Lemma \ref{l-reformulate-L^2}). Next, to examine the decay rate, we analyze the spectrum of the linear operator $\Linearpart:\ell^2\rightarrow\ell^2$ in the following proposition, with its proof provided in Appendix \ref{appendix-3}.  

\begin{proposition}\label{p-spectrum}
Denote the spectrum of the linear operator $\Linearpart$ defined in \eqref{eq-matrix-L} by $\operatorname{sp}(\Linearpart)$. Then,  $\operatorname{sp}(\Linearpart)=\{0\}\cup\{\lambda\in\mathbb{C}:\mathbb{M}(\lambda)=0\}$. The eigenvalue $\lambda=0$ has multiplicity one, and $\mathbb{M}(\lambda)$ is a meromorphic function defined by 
\begin{align}
\mathbb{M}(\lambda)={}&\frac{\pi}{\GASCONST\Tconst\Adi}\Paren{\frac{4}{3}+\sum_{k=1}^\infty\frac{8(\Adi-1)\bkappa}{\pi^2\bkappa k^2+\lambda}}\Bracket{\ldens\tR\lambda^2+\frac{4\mu_l\bV}{\ssr(\ssr^3+\bV)}\lambda-\Paren{\frac{2\sigma}{\ssr^2}+\frac{2\barsigma\ssr^2}{\ssor^{4}}}}\nonumber\\
&+4\pi\frac{\ssd}{\ssr},\label{eq-meromorphic-function}
\end{align} 
where $\tR=\ssr-{\ssr^2}/{\ssor}$ and $\bkappa=\Thermal/(\ssr^2\ssd\Adi \Heatcapa)$ as in Proposition \ref{p-infinite-ODE}.

Moreover, there exists a constant $\varpi>0$, such that $\sup\{\operatorname{Re}(\lambda):\lambda\in\operatorname{sp}(\Linearpart)\setminus\{0\}\}\le-2\varpi<0$.
\end{proposition}   

By further analyzing the spectrum bound $\sup\{\operatorname{Re}(\lambda):\lambda\in\operatorname{sp}(\Linearpart)\setminus\{0\}\}$, we discover that the characterization of the operator spectrum is more accurate when the liquid volume is small.

More precisely, given the mass of the gas $M>0$, we assume that the liquid volume $V>0$ is sufficiently small. Then, using the notations in Appendix \ref{appendix-3} and applying \eqref{eq-equi-eq}, we have 
\begin{align*}
B^2-4KC^2&=\Paren{\frac{4\mu_l\bV}{\ssr(\ssr^3+\bV)}}^2-\frac{8\GASCONST\Tconst\ssd\ssor}{\ldens\ssr^2(\ssor-\ssr)}\Paren{\ldens\frac{\ssr\ssor-\ssr^2}{\ssor}}^2\\
&=\frac{16\mu_l^2}{\ssr^2(\ssr^3+\bV)^2}\bV^2-16\ldens\Paren{\frac{\sigma}{\ssr}+\frac{\barsigma}{\ssor}}\Paren{1-\frac{\ssr}{\ssor}}\\
&\le\frac{16\mu_l^2}{\ssr^8}\bV^2-\frac{16\ldens\sigma}{\ssr}\Paren{1-\Paren{1+\bV/\ssr^3}^{-1/3}}\\
&\le\frac{16\mu_l^2}{\ssr^8}\bV^2-\frac{16\ldens\sigma}{\ssr}\Paren{\frac{3\bV}{\ssr^3}+O\Paren{\bV^2}}<0,
\end{align*}
where we have used the lower bound of the equilibrium radius $\ssr$ in \eqref{eq-Rstar-bound}. Then, from the results in Lemma \ref{l-real-part}, direct calculations show that $K=\frac{2\GASCONST\Tconst\ssd\ssor}{\ldens\ssr^2(\ssor-\ssr)}$ is large, and 
\begin{equation*}
\varpi=\frac{1}{2}\min\Brace{\Theta_2\pi^2\bkappa,\max\Brace{\frac{6\mu_l}{\ldens\ssr^2\ssor^2}\frac{1}{1+O(V)},\Theta_1\pi^2\bkappa}}=\frac{1}{2}\Theta_0\pi^2\bkappa,
\end{equation*}
where the constant $\Theta_0$ satisfying $0<\Theta_1\le\Theta_0\le\Theta_2<1$. However, by inspecting the form of $\mathbb{M}(\lambda)$, it has a pole (nearest to the origin, i.e., choosing $k=1$) at $-\pi^2\bkappa$. For $\lambda$ near this pole, since $\bV>0$ is sufficiently small, it holds that the following term in the bracket of \eqref{eq-meromorphic-function}
\begin{equation*}
\ldens\tR\lambda^2+\frac{4\mu_l\bV}{\ssr(\ssr^3+\bV)}\lambda-\Paren{\frac{2\sigma}{\ssr^2}+\frac{2\barsigma\ssr^2}{\ssor^{4}}}\ \text{is close to}\ -\frac{2\sigma}{\ssr^2}-\frac{2\barsigma\ssr^2}{\ssor^{4}}<0.
\end{equation*}
Consequently, $\mathbb{M}(\lambda)$ tends to $-\infty$ as $\lambda$ approaches the pole $-\pi^2\bkappa$ from the right side on the real axis. According to the estimates of \textbf{Case 3} in Lemma \ref{l-real-part}, $\mathbb{M}(\lambda)>0$ on $\{\lambda\in\mathbb{R}:\lambda\ge 0\}$. This observation shows that at least on the interval $(-\pi^2\bkappa,0)$, $\mathbb{M}(\lambda)$ has a real root. Thus, we find both the lower and upper bounds
\begin{equation*}
-\pi^2\bkappa<\sup\{\operatorname{Re}(\lambda):\lambda\in\operatorname{sp}(\Linearpart)\setminus\{0\}\}\le- \Theta_0\pi^2\bkappa.
\end{equation*}

From the center manifold theory in Appendix \ref{appendix-4} and main Theorem \ref{t-main result}, the index $-\varpi_\dagger$ in the convergence rate $e^{-\varpi_\dagger t}$ is determined by this spectrum bound (difference by a constant). Therefore, it is reasonable to anticipate that the convergence rate of the spherically symmetric solution will be determined by the magnitude of $-\pi^2\bkappa$. Utilizing the radius bound \eqref{eq-Rstar-bound} again, we see that
\begin{equation*}
-\pi^2\bkappa=-\frac{\Thermal\pi^2}{\ssr^2\ssd\Adi\Heatcapa}=-\frac{4\Thermal\pi^3\ssr}{3M\Adi\Heatcapa}\in\frac{\Thermal\pi^2}{\Adi\Heatcapa}\sqrt{\frac{2\pi\GASCONST}{3\sigma}}\Paren{-\sqrt{\frac{\Tconst}{M}},-\sqrt{\frac{\Tconst}{M}}\bigg/\sqrt{1+\frac{\barsigma}{\sigma}}}.
\end{equation*}
We conclude that for small liquid volumes, a reduction in gas mass or an increase in temperature can accelerate convergence. This effect has not been observed when the liquid volume is infinite.

\noindent \textbf{Step 3:}  As we have demonstrated the asymptotic stability in the space $\ell^2$ (cf. Lemma \ref{l-reformulate-L^2}), we will decompose the space into the direct sum $\ell^2=X\oplus Y$ and then derive a corresponding system, \eqref{eq-dotx-doty} below, which is equivalent to systems \eqref{eq-IBVP} and \eqref{eq-Lw+Nw} in Proposition \ref{p-infinite-ODE}, as well as the original free boundary problem \eqref{eq-main-syst}.

Following the setup in Appendix \ref{appendix-4}, we first observe that according to Proposition \ref{p-spectrum}, the eigenvalue $\lambda=0$ has multiplicity one. Therefore, the linear operator $\Linearpart$ defined in \eqref{eq-matrix-L}  
has a one-dimensional kernel $X=\ker\Linearpart=\textup{span}(\UU)$, where 
\begin{equation}\label{eq-right-eigenvector}
\UU=\begin{bmatrix} -{2\sigma}\big/\Paren{\GASCONST\Tconst\ssr^2}-{2\barsigma\ssr^2}\big/\Paren{\GASCONST\Tconst\ssor^{4}}&1&0&0&0&\cdots 
\end{bmatrix}^\top.
\end{equation}

Moreover, the vector $\Up^\top=\begin{bmatrix}{4\pi}/3& 4\pi\ssd/{\ssr}&0&{\Adi\coeff_1}/\Paren{\Adi-1}&{\Adi\coeff_2}/\Paren{\Adi-1}&\cdots\end{bmatrix}$ is the corresponding left eigenvector of $\Linearpart$ once we note that  $\sum_{j=1}^\infty\coeff_j^2=4(\Adi-1)^2\pi/\Paren{3\Adi^2}$. 

Applied the equilibrium algebraic equations \eqref{eq-equi-eq}, it follows that
\begin{equation*}
\Anglebracket{\Up^\top,\UU}=\frac{4\pi}{3\GASCONST\Tconst}\Paren{-\frac{2\sigma}{\ssr^2}-\frac{2\barsigma\ssr^2}{\ssor^{4}}+\frac{3\GASCONST\Tconst\ssd}{\ssr}}=\frac{4\pi}{3\GASCONST\Tconst}\Paren{\frac{4\sigma}{\ssr^2}+\Paren{\frac{6\barsigma}{\ssor\ssr}-\frac{2\barsigma\ssr^2}{\ssor^{4}}}},
\end{equation*} 
where we notice that the term ${6\barsigma}/{\ssor\ssr}-{2\barsigma\ssr^2}/{\ssor^{4}}$ is strict positive since $\ssr<\ssor$. Then, we can normalize $\Up^\top$ such that $\langle\Up^\top_0,\UU\rangle=1$ by setting $\Up^\top_0=\Anglebracket{\Up^\top,\UU}^{-1}\Up^\top$.

Having obtained both the left and right eigenvectors, we decompose $\ell^2=X\oplus Y$ as follows: 
$\zz=\xx+\yy$, where $\xx=\mathcal{Q}_1\zz=(\Up^\top_0\zz)\UU\in X$ and 
$\yy=\mathcal{Q}_2\zz=\zz-\mathcal{Q}_1\zz\in Y$. 

Since $\Linearpart\UU=\Up^\top\Linearpart=\Zero$, we have $0=\Up^\top_0\Linearpart\zz\UU=\mathcal{Q}_1\Linearpart\zz$, $\Linearpart\mathcal{Q}_1\zz=(\Up^\top_0\zz)\Linearpart\UU=\Zero$, and $\mathcal{Q}_2\Linearpart\zz=\Linearpart\mathcal{Q}_2\zz=\Linearpart\yy=\Linearpart\zz$. 
In particular, $\Linearpart|_X=\Zero$, and $\mathcal{Q}_2\Linearpart\big|_Y=\Linearpart\mathcal{Q}_2\big|_Y$ is the restriction of $\Linearpart$ on $Y$ satisfying $\sup\{\operatorname{Re}(\lambda):\lambda\in\operatorname{sp}(\Linearpart\big|_{Y})\}\le-2\varpi<0$. 

Then, we derive a dynamical system of $(\xx,\yy)$ from \eqref{eq-Lw+Nw},  which is of the same form as system \eqref{eq-app-4-3}. That is,
\begin{equation}\label{eq-dotx-doty}
\begin{cases}
\dot\xx=\mathcal{Q}_1\Nonlpart(\xx +\yy,\dot\xx+\dot\yy)\\
\ \ {}=\mathcal{Q}_1\Bracket{\Nonlpart^1(\xx+\yy)[\dot\xx+\dot\yy]}+\mathcal{Q}_1\Nonlpart^0(\xx+\yy),\ &t>0,\\ 
\dot\yy=\Linearpart\yy+\mathcal{Q}_2\Nonlpart(\xx+\yy,\dot\xx+\dot\yy)\\
\ \ {}=\Linearpart\yy+\mathcal{Q}_2\Bracket{\Nonlpart^1(\xx+\yy)[\dot\xx+\dot\yy]}+\mathcal{Q}_2\Nonlpart^0(\xx+\yy),\ &t>0. 
\end{cases}
\end{equation}  
The initial condition can be deduced from \eqref{eq-uRZ-initial}. 

\noindent \textbf{Step 4:} We verify the remaining requirements in (i)--(iii) in Appendix \ref{appendix-4}, especially the decay estimates in \eqref{eq-app-4-2}. Since the one-dimensional subspace  $X$ is $\Linearpart$-invariant ($\Linearpart|_X=\Zero$) and $Y$ is closed, we need to check the following results.
\begin{proposition}\label{p-invariant space}
The subspace $Y$ is $e^{\Linearpart t}$-invariant, and for any $t\ge 0$, $\norm{e^{\Linearpart t} \mathcal{Q}_2}_{\ell^2}\le ce^{-\varpi t}$, where the constant $c>0$, and the index $\varpi>0$ is given in Lemma \ref{l-real-part}.
\end{proposition}
\begin{proof} 
Given any $\yy_0\in Y$, one has
\begin{equation*}
e^{\Linearpart t}\yy_0=\sum_{n=0}^{\infty}\frac{(t\Linearpart)^n}{n!}\yy_0=\mathcal{Q}_2\yy_0+\sum_{n=1}^{\infty}\frac{(t\mathcal{Q}_2\Linearpart)^n}{n!}\yy_0=\mathcal{Q}_2\Paren{y_0+t\Linearpart\sum_{n=1}^{\infty}\frac{(t\mathcal{Q}_2\Linearpart)^{n-1}}{n!}\yy_0},
\end{equation*}
where we have used the fact that for $n\ge 1$,  $\Linearpart^{n-1}\Linearpart\yy_0=\Linearpart^{n-1}\mathcal{Q}_2\Linearpart\yy_0=\Linearpart^{n-2}\Linearpart\mathcal{Q}_2\Linearpart\yy_0=\Linearpart^{n-2}\mathcal{Q}_2\Linearpart\mathcal{Q}_2\Linearpart\yy_0=\Paren{\mathcal{Q}_2\Linearpart}^{n}\yy_0$. 
Therefore, we deduce $e^{\Linearpart t}\yy_0\in Y$, and we conclude that $Y$ is $e^{\Linearpart t}$-invariant. The operator estimate $\norm{e^{\Linearpart t} \mathcal{Q}_2}$ follows from the boundedness of $\mathcal{Q}_2$ and the spectrum analysis in Proposition \ref{p-spectrum}, together with the Gearhart-Pr\"uss theorem 
\cite{Gearhart1978,Pruess1984} 
for $C_0$ semigroups. 
\end{proof} 

\noindent \textbf{Step 5:} In the following, we show the existence of a global center manifold for \eqref{eq-dotx-doty}, and verify the Lyapunov stability of the zero solution to the equation on the center manifold, as required in Lemma \ref{l-e-converge}. 

Compared to the local center manifold constructed in \cite[Lemma 9.6]{Lai2023}, it is worth mentioning that the manifold of equilibria $\Sigma$ given in \eqref{eq-Mstar} is actually a global center manifold since we do not utilize the smallness assumption in this part.
\begin{proposition}\label{p-center}
Given any $(\ssd,\ssr)=(\ssd[M,V],\ssr[M,V])\in \Sigma$, for $\alpha\in\mathbb{R}$, 
we define $\hatd(\alpha)$ by
\begin{equation}\label{eq-center-density}
\hatd(\alpha)=\frac{1}{\GASCONST\Tconst}\Paren{{2\sigma}\big/{\hatr(\alpha)}+{2\barsigma}\Big/\sqrt[3]{\hatr(\alpha)^3+\bV}},
\end{equation} 
where $\hatr$ is an arbitrary positive $C^1$ function such that 
$\hatr(0) =\ssr$ and $ \hatr^{\prime}(0)\neq 1$. 

Then, it holds $\hatd(0)=\ssd$. Denote $\xx\in X$ by $\xx=\alpha\UU,\alpha\in \mathbb{R}$, and define the curve $\yy=h(\xx)=h(\alpha\UU)$ by
\begin{equation*}
h(\xx)=\begin{bmatrix}
\hatd(\alpha)+\Paren{\dfrac{2\sigma}{\GASCONST\Tconst\ssr^2}+\dfrac{2\barsigma\ssr^2}{\GASCONST\Tconst\ssor^{4}}}\alpha-\ssd&\hatr(\alpha)-\alpha-\ssr&0&0&\cdots
\end{bmatrix}^\top.
\end{equation*} 
It follows that the curve $h$ is a global center manifold for system \eqref{eq-dotx-doty}.

Moreover, for $\xx(t)=\alpha(t)\UU$ with $|\alpha(t)|$ small enough, the equation on the center manifold given by
\begin{equation}\label{eq-equation-on-center}
\dot\xx=\mathcal{Q}_1\Bracket{\Nonlpart^1(\xx+h(\xx))[\dot\xx+h^\prime(\xx)\dot\xx]}+\mathcal{Q}_1\Nonlpart^0(\xx+h(\xx)),
\end{equation} 
is trivial. That is, \eqref{eq-equation-on-center} is equivalent to $\dot\alpha\equiv0$. Therefore, the zero solution to \eqref{eq-equation-on-center} is Lyapunov stable.
\end{proposition} 
\begin{proof}
Without loss of generality, for any $(\xx(0),h(\xx(0)))$, 
we consider the solution in the form of $(\xx(t)=\alpha(t)\UU,\yy(t))$ to system \eqref{eq-dotx-doty}. Note that
the initial data yield the initial condition $\zz(0)=\xx(0)+h(\xx(0))=
\begin{bmatrix}
\hatd(\alpha)-\ssd&\hatr(\alpha)-\ssr&0&0&\cdots
\end{bmatrix}^\top$.

Recalling in Proposition \ref{p-infinite-ODE}, $\zz= \begin{bmatrix}\Errorrho_{2}&\ErrorR&\ptErrorR&\theta_1&\theta_2&\cdots\end{bmatrix}^\top$, combining with the changing variables \eqref{eq-change-var} and the decomposition \eqref{eq-u-expansion}, we deduce that $\theta_i(0)=0$ and therefore $\Errorrho_{1}(\cdot,0)\equiv0$. Also, it holds $R(0)=\hatr(\alpha(0)),\Errorrho_{2}(0)=\hatd(\alpha(0))-\ssd$, and $\gdens(R(0)y,0)=\ssd+\Errorrho_{2}(0)=\hatd(\alpha(0))$ for any $y\le1$. That is, $\gdens(\cdot,0)\equiv\hatd(\alpha(0))$.

From the trajectory defined in \eqref{eq-center-density}, these initial data are exactly the equilibrium of the gas-liquid system with the gas mass  
\begin{equation*}
M=\frac{4\pi}{3}\hatd(\alpha(0))\hatr(\alpha(0))^3
\end{equation*}
and the liquid volume $V$.
Therefore, the global-in-time solution is
\begin{equation*}
(\xx(t),\yy(t))\equiv(\xx(0),\yy(0))=(\xx(0),h(\xx(0))).
\end{equation*} 
In other words, it follows that 
\begin{equation*}
\gdens(x,t)\equiv\hatd(\alpha(0)), x\in B_{\hatr(\alpha(0))}, R(t)\equiv \hatr(\alpha(0)),\ \text{and}\ \alpha(t)\equiv\alpha(0).
\end{equation*}

Next, we verify that the curve $\yy=h(\xx)$ is tangent to the subspace $X$ at the origin by differentiating \eqref{eq-center-density}
\begin{equation}\label{eq-partial-rho-s-s}
\hatd^\prime(\alpha)=-\frac{1}{\GASCONST \Tconst}\Bracket{ \frac{2\sigma}{ \hatr(\alpha)^2}+\frac{2\barsigma \hatr(\alpha)^2 }{\Paren{ \hatr(\alpha)^3+\bV}^{4/3}} }\hatr^\prime(\alpha).
\end{equation} 
It follows that 
\begin{equation*}
\hatd^\prime(0)
=-\Paren{\frac{2\sigma}{\GASCONST\Tconst\ssr^2}+\frac{2\barsigma\ssr^2}{\GASCONST\Tconst\ssor^{4}}}\hatr^\prime(0),
\end{equation*} 
by setting $\alpha=0$ and using the assumption $\hatr(0)=\ssr$. Recalling the vector $\UU$ defined in \eqref{eq-right-eigenvector}, we conclude that
\begin{equation*}
\frac{dh(\alpha\UU)}{d\alpha} \bigg|_{\alpha=0}=\begin{bmatrix}
-\Paren{\dfrac{2\sigma}{\GASCONST\Tconst\ssr^2}+\dfrac{2\barsigma\ssr^2}{\GASCONST\Tconst\ssor^{4}}}\Paren{\hatr^\prime(0)-1}&\hatr^\prime(0)-1&0&\cdots
\end{bmatrix}^\top\in X. 
\end{equation*} 

Finally, we check that equation \eqref{eq-equation-on-center} on the center manifold is trivial, provided $\xx(t)=\alpha(t)\UU$ with $|\alpha(t)|$ sufficiently small. Since $\yy(t)=h(\xx(t))$ on the center manifold, we have
\begin{equation}\label{eq-x+h(x)}
\begin{cases}
\zz(t)=\xx(t)+h(\xx(t))= 
\begin{bmatrix}
\hatd(\alpha(t))-\ssd&\hatr(\alpha(t))-\ssr&0&0&\cdots
\end{bmatrix}^\top,\\ 
\dot\zz(t)=\dot\xx(t)+h^\prime(\xx(t))\dot\xx(t)=\dot\alpha(t)
\begin{bmatrix}
\hatd^\prime(\alpha(t))&\hatr^\prime(\alpha(t))&0&0&\cdots
\end{bmatrix}^\top. 
\end{cases}
\end{equation} 
Thus, for the nonlinear terms $\Pi^0,\bPi^1,\Phi^0,\bPhi^1,\Psi^0$, and $\bPsi^1$ defined in \eqref{eq-def-nonl-term}, from $\ptErrorR\equiv0,\theta_i\equiv0$, and $\Errorrho_{1}\equiv0$, we deduce that  
\begin{equation*}
\begin{cases}
\Pi^0(\xx(t)+h(\xx(t)))=0,\quad \bPi^1(\xx(t)+h(\xx(t)))=\Zero,\\
\Phi^0(\xx(t)+h(\xx(t)))=0,\quad \bPhi^1(\xx(t)+h(\xx(t)))=\begin{bmatrix}g(\alpha(t))&0& 0&\!\!\cdots\end{bmatrix}^\top\!\!,\\
\Psi^0(\xx(t)+h(\xx(t)))=\frac1{\GASCONST\Tconst}\bigg[\frac{2\sigma}{\ssr^2}\Paren{\hatr(\alpha(t))-\ssr}+\frac{2\sigma}{\hatr(\alpha(t))}-\frac{2\sigma}{\ssr}\\ \qquad\qquad\qquad\qquad\qquad\quad \ \  +\frac{2\barsigma\ssr^2}{\ssor^{4}}\Paren{\hatr(\alpha(t))-\ssr}+\frac{2\barsigma}{\Paren{\hatr(\alpha(t))^3+\bV}^{\frac13}}-\frac{2\barsigma}{\ssor}\bigg]\ \ ,\\
\bPsi^1(\xx(t)+h(\xx(t)))=\begin{bmatrix}0&0&\frac{\ldens}{\GASCONST\Tconst}\Paren{\hatr(\alpha(t))-\ssr+\frac{\ssr^2}{\ssor}-\frac{\hatr(\alpha(t))^2}{\Paren{\hatr(\alpha(t))^3+\bV}^{\frac13}}}&0&\!\!\cdots\end{bmatrix}^\top\!\!,
\end{cases}
\end{equation*}
where
\begin{equation}\label{eq-def-g(alpha)}
g(\alpha(t))
=\frac{\ssr}{3\Adi\ssd}\Paren{1-\frac{\ssd\hatr(\alpha(t))}{\hatd(\alpha(t))\ssr}}.
\end{equation} 
By \eqref{eq-def-N1N0}, the nonlinear term $\Nonlpart^1(\xx(t)+h(\xx(t)))\Bracket{\dot\xx(t)+h^\prime(\xx(t))\dot\xx(t)}$ equals 
\begin{equation*}
3\Adi\ssd\ssr^{-1}\hatd^\prime(\alpha(t))g(\alpha(t))\dot\alpha(t) \begin{bmatrix}
1&0&0&-\coeff_1&-\coeff_2&\cdots \end{bmatrix}^\top,
\end{equation*} 
and $\Nonlpart^0(\xx(t)+h(\xx(t)))=\begin{bmatrix}0&0&-\frac{\GASCONST\Tconst}{\ldens\tR}\Psi^0(\xx(t)+h(\xx(t)))&0&0&\cdots\end{bmatrix}^\top$. Then, noting that $\Anglebracket{\Up^\top_0,\Nonlpart^0(\xx(t)+h(\xx(t)))}=0$, equation \eqref{eq-equation-on-center} is further equivalent to
\begin{align*}
\dot\xx&=\Anglebracket{\Up^\top_0,\Nonlpart^1(\xx(t)+h(\xx(t)))[\dot\xx+h^\prime(\xx(t))\dot\xx]}\UU\\
&=4\pi\ssd\ssr^{-1}\Anglebracket{\Up^\top,\UU}^{-1}\hatd^\prime(\alpha(t))g(\alpha(t))\dot\alpha\UU,
\end{align*}
where we have used the identity $\sum_{j=1}^\infty\coeff_j^2=\frac{4(\Adi-1)^2\pi}{3\Adi^2}$ and the left eigenvector $\Up^\top$. Combining \eqref{eq-def-g(alpha)} and \eqref{eq-partial-rho-s-s}, we conclude that
\begin{align*}
&\dot\alpha\Bracket{1-4\pi\ssd\ssr^{-1}\Anglebracket{\Up^\top,\UU}^{-1}\hatd^\prime(\alpha(t))g(\alpha(t))}\\
={}&\dot\alpha\Bracket{1+\frac{4\pi\hatr^\prime(\alpha(t))}{3\Anglebracket{\Up^\top,\UU}\GASCONST\Tconst\Adi}\Paren{\frac{2\sigma}{\hatr(\alpha(t))^2}+\frac{2\barsigma\hatr(\alpha(t))^2}{(\hatr(\alpha(t))^3+\bV)^{\frac43}}}\Paren{1-\frac{\ssd\hatr(\alpha(t))}{\hatd(\alpha(t))\ssr}}}\\
\triangleq{}&\dot\alpha\Brace{1+\Bracket{1-\ssd\hatr(\alpha(t))/\Paren{\hatd(\alpha(t))\ssr}}K(\alpha(t))}
\end{align*} 
vanishes, since $\dot\xx(t)=\dot\alpha(t)\UU$ and $\UU\neq \Zero$. Note that for $\module{\alpha(t)}$ small enough, $K(\alpha(t))$ is bounded and the factor $1-\ssd\hatr(\alpha(t))/\Paren{\hatd(\alpha(t))\ssr}$
is sufficiently small. This yields that $\dot\alpha(t)=0$ for all $\module{\alpha(t)}$ sufficiently small. In other words, the dynamic on the center manifold is trivial, and the proof is completed. 
\end{proof}  

\noindent \textbf{Step 6:} The assumptions for the nonlinear terms in system \eqref{eq-Lw+Nw} are verified in the following lemma. 
\begin{lemma}\label{l-verify-N(W,P)}
For the nonlinear term $\Nonlpart(\zz,\ww)$ defined in \eqref{eq-N(w,}, where $\zz$ is computed by \eqref{eq-change-var} and \eqref{eq-u-expansion} from the solution $(\brho,R)
$ of problem \eqref{eq-app-2}. 
Then, we have $\Nonlpart(\zz,\ww)\in\ell^2$, $\Nonlpart(\Zero,\cdot)\equiv\Zero$, $\parl_\ww\Nonlpart(\Zero,\Zero)=\Zero$, and $\parl_\zz\Nonlpart(\Zero,\Zero)=\Zero$.
\end{lemma}
\begin{proof} 
The fact that $\Nonlpart(\zz, \ww) \in \ell^2$ follows from the same arguments in \cite[Proposition 9.7]{Lai2023}. Also, it is clear that $\Nonlpart(\Zero,\ww)\equiv\Zero$ for all $\ww$ by using the definition in \eqref{eq-def-nonl-term}.
To compute the partial derivatives, we utilize the decomposition $\Nonlpart(\zz,\ww)=\Nonlpart^1(\zz)\ww+\Nonlpart^0(\zz)$ from \eqref{eq-N1P+N0}. Applying \eqref{eq-def-nonl-term} again, we have $|\bPi^1(\zz)|+|\bPhi^1(\zz)|+|\bPsi^1(\zz)|\le O(\norm\zz)$ and $|F^0(\zz)|+|\Phi^0(\zz)|+|\Psi^0(\zz)|\le O(\norm\zz^2)$. To see this, for terms $\bPsi^1(\zz)$ and $\Psi^0(\zz)$, using Taylor's theorem at the equilibrium radius $\ssr$ (e.g., ${\bR^{-1}}={\ssor^{-1}}-{\ssr^2}{\ssor^{-4}}\ErrorR+O(\ErrorR^2)$), one has
\begin{equation*}
\frac{\ldens}{\GASCONST\Tconst}\module{\ErrorR+\frac{\ssr^2}{\ssor}-\frac{R^2}{\bR}}=O(\norm{\zz}),\quad\text{as}\ \norm{\zz}\rightarrow 0.
\end{equation*}
Therefore, $|\bPsi^1(\zz)|\le O(\norm\zz)$, as $\norm{\zz}\rightarrow 0$. Also, from  the ratio ${R}/{\bR}<1$, it holds
\begin{align*} 
\module{\Psi^0(\zz)}={}&\left|\frac1{\GASCONST\Tconst}\left[4\mu_l\bV\Paren{\frac{1}{R(R^3+\bV)}-\frac{1}{\ssr(\ssr^3+\bV)}}\dot\ErrorR+{\frac{2\sigma\ErrorR^2}{\ssr^2(\ssr+\ErrorR)}}\right.\right.\\
&\qquad\ \ \left.\left.+2\barsigma\Paren{\frac{\ssr^2}{\ssor^{4}}\ErrorR+\frac{1}{\bR}-\frac{1}{\ssor}}\right]+\frac{\ldens}{\GASCONST\Tconst}\Paren{\frac32-\frac{2R}{\bR}+\frac{R^4}{2\bR^4}}\dot\ErrorR^2\right|\\
\le{}&O\Paren{\lVert\ErrorR\dot\ErrorR\rVert+\lVert\ErrorR\rVert^2+\lVert\dot\ErrorR\rVert^2}\le O(\norm\zz^2),\quad\text{as}\ \norm{\zz}\rightarrow 0.
\end{align*} 
We conclude that $\Nonlpart^1(\Zero)=\Nonlpart^0(\Zero)=\Zero$, and $\parl_\zz\Nonlpart^0(\Zero)=\Zero$ by applying \eqref{eq-def-N1N0}. This completes the proof.	
\end{proof}   

\noindent \textbf{Step 7:} By incorporating the aforementioned preparations, we close the proof of the nonlinear exponential decay rate.  

We recall that free boundary problem \eqref{eq-main-syst} is equivalent to system \eqref{eq-app-2} defined in $B_1$, and is also equivalent to dynamical system \eqref{eq-Lw+Nw} by Proposition \ref{p-infinite-ODE}. 

Applying part (i) in Theorem \ref{t-main result} and Lemma \ref{l-reformulate-L^2}, we obtain $|\pttR(t)|+|\dddot R(t)|\rightarrow 0$ and $\zz(t)\rightarrow\Zero$ as $t\rightarrow\infty$. We also deduce that $\parl_t\brho\rightarrow 0$ uniformly from \eqref{eq-app-2-1}. Combining these results, $\norm{\dot\zz(t)}_{\ell^2}\rightarrow\Zero$ follows. Without loss of generality, we may assume that $\norm{\dot\zz(t)}$ is small enough for all time. Applying Lemma \ref{l-e-converge} to  system \eqref{eq-Lw+Nw}, there exists $\alpha\in\mathbb{R}$ with $\module{\alpha}$ small, such that 
\begin{equation*}
\norm{\xx(t)-\alpha\UU+\yy(t)-h(\alpha\UU)}_{\ell^2}= O(e^{-\varpi_0 t}),\quad\text{as}\ t\rightarrow\infty.
\end{equation*} 
In other words, from the first equation in \eqref{eq-x+h(x)}, we obtain
\begin{equation*}
\norm{\begin{bmatrix}\Errorrho_{2}&\ErrorR&\ptErrorR&\theta_1&\theta_2&\cdots\end{bmatrix}-\begin{bmatrix}\hatd(\alpha)-\ssd&\hatr(\alpha)-\ssr&0&0&\cdots\end{bmatrix}}_{\ell^2}=O(e^{-\varpi_0t}),
\end{equation*} 
as $t\rightarrow\infty$. Clearly, $(\hatd(\alpha),\hatr(\alpha))$ coincides with the equilibrium $(\ssd[M,V],\ssr[M,V])$, since $\zz(t)\rightarrow\Zero$. This yields the exponential convergence of $(\Errorrho,\Errorrho_{1},\ErrorR,\ptErrorR)$
\begin{equation*}  \norm{\Errorrho_{1}}_{L^2(B_1)}^2+(\gdens(R(t),t)-\ssd[M,V])^2+(R(t)-\ssr[M,V])^2+\ptR(t)^2=O\Paren{e^{-2\varpi_0 t}}, 
\end{equation*} 
as $t\rightarrow\infty$, where we have used the changing variables in \eqref{eq-change-var}. Again from \eqref{eq-change-var}, it holds $\gdens(R(t)y,t)-\ssd[M,V]=\Errorrho_{1}(y,t)+\gdens(R(t),t)-\ssd[M,V]$,
and we deduce that
\begin{equation*}  
\norm{\gdens(R(t)y,t)-\ssd[M,V]}_{L^2_y(B_1)}^2=O\Paren{e^{-2\varpi_0t}},\quad \text{as}\ t\rightarrow\infty,
\end{equation*}
by triangle inequality. 
The remaining proof is similar to the $W^{1,\infty}$ estimates and the $C^{2+2\alpha}$ decay estimates in \cite[Lemma 9.8 and Proposition 9.9]{Lai2023}. This completes the proof of Theorem \ref{t-main result}.  
\appendix
\section{Regular spherical equilibria to system \eqref{eq-l-full}--\eqref{eq-lo-full}}\label{appendix-1}

In this appendix, we show that the regular spherically symmetric equilibrium solutions to the original full free boundary problem \eqref{eq-l-full}--\eqref{eq-lo-full} are determined by the mass-volume pairs, provided that the liquid temperature remains constant.
\begin{proposition}\label{p-full} 
Any regular spherical equilibrium solution to system \eqref{eq-l-full}--\eqref{eq-lo-full} with mass (of the gas) $M$ and liquid volume $V$ satisfying  $T_l\equiv\Tconst$ is given by
\begin{align*}
&\VV_{l}\equiv\Zero,\quad\VV_{g}\equiv\Zero,\quad p_{l}\equiv \frac{2\barsigma}{\ssor[M,V]},\quad\Omn\equiv B_{_{\ssr[M,V]}},\quad
\overline{\Omn}\cup\Omp\equiv B_{_{\ssor[M,V]}},\\
&\ssd[M,V]\equiv \frac2{\GASCONST \Tconst}\Paren{\frac{\sigma}{\ssr[M,V]}+\frac{\barsigma}{\ssor[M,V]}},\quad p_{g}\equiv 2\Paren{\frac{\sigma}{\ssr[M,V]}+\frac{\barsigma}{\ssor[M,V]}},\\
&T_{g}\equiv\Tconst,\quad T_{l}\equiv\Tconst,\quad s\equiv\Heatcapa\log\Paren{(\GASCONST\Tconst)^\Adi\Paren{\frac{2\sigma}{\ssr[M,V]}+\frac{2\barsigma}{\ssor[M,V]}}^{1-\Adi}}, 
\end{align*}
where $\ssr\in C^\infty((0,\infty)^2;(0,\infty))$ is the smooth map denoting the equilibrium radius of the gas bubble as defined in Theorem \ref{t-equilibria}
, $\ssor[M,V]=\sqrt[3]{\ssr[M,V]^3+\bV}$ represents the external radius of the entire gas-liquid region and $\bV=3V/4\pi$ is the modified liquid volume.
\end{proposition}
\begin{proof}
Assume that $\VV_{l}(x)=v_{l}(r)x/r$ with $r=\module{x}$. The divergence-free condition reads $\parl_rv_{l}(r)+(2/r)v_{l}(r)=0$, or $\parl_r(r^2 v_{l}(r))=0,\ssr\le r\le\ssor$. Therefore, $v_{l}(r)=a/r^2,\ssr\le r\le\ssor$ for some constant $a$. However, the boundary condition \eqref{eq-lg-full-1} implies $v_{l}(\ssr)=0$. Thus, $v_{l}\equiv0$, and \eqref{eq-l-full-1}  becomes $\nabla p_{l}=0$. From \eqref{eq-lo-full-2}, we conclude that the pressure $p_{l}\equiv2\barsigma/\ssor$.

For the gas velocity $v_{g}$, \eqref{eq-g-full-1} implies $\parl_r(r^2\gdens v_{g})=0$ for $r\le\ssr$. Therefore, $r^2\gdens v_{g}$ is a constant.
Again by \eqref{eq-lg-full-1}, $v_{g}(\ssr)=0$ follows and $\gdens v_{g}\equiv0$. Since we consider the regular solution, $\gdens \neq0$ by  \eqref{eq-g-full-5}. Therefore, $\VV_{g}\equiv\Zero$ and $p_{g}$ is a constant from \eqref{eq-g-full-2}. Now that $v_{l}=v_{g}\equiv0$, \eqref{eq-lg-full-2} yields $p_{g}\equiv2\sigma/\ssr+p_{l}=2\sigma/\ssr+2\barsigma/\ssor$. Moreover, \eqref{eq-g-full-3} becomes $\Laplace T_{g}=0$ in $B_{\ssr}$. Since $T_{g}$ is regular, we have $T_{g}\equiv T_{g}(\ssr)=\Tconst$ by \eqref{eq-lg-full-3} and the maximum principle. 

Finally, by \eqref{eq-g-full-4}, $\gdens\equiv\Paren{2\sigma/\ssr+2\barsigma/\ssor}/\Paren{\GASCONST \Tconst}$.
Due to the conservation of mass \eqref{eq-conserve}, $M=\frac{4\pi}3\gdens\ssr^3$. These imply that the spherically symmetric equilibrium $(\gdens,\ssr)$ can be obtained by solving the algebraic equation \eqref{eq-equi-eq}. Therefore, the proposition follows from the same arguments as in the proof of Theorem \ref{t-equilibria}.
\end{proof}  
\section{Well-posedness and Lyapunov stability of problem \eqref{eq-main-syst}}\label{appendix-2}

To state the well-posedness and the Lyapunov stability results, given any gas density and bubble radius $(\gdens(\cdot,t),R(t))$, we define $\brho(y,t)=\gdens(R(t)y,t),y\in\overline{B_1}$ for $t>0$, and introduce the norm $\norm{\cdot}_{C^{2+2\alpha}(B_A)}$ as follows. For a function $f(r),r<A$, where $A>0$ is a constant, we define a radial function $\tilde{f}(x)=f(|x|)$ for $x\in B_A$ and
\begin{equation*}
\norm{f}_{C^{2+2\alpha}_r}=\norm{\tilde{f}}_{C^{2+2\alpha}_x}=\max_{|\beta|\le 2}\sup_{x\in B_A}\module{D^\beta\tilde{f}(x)}+\sup_{x_1\neq x_2\in B_A}\frac{\module{D^2\bar{f}(x_1)-D^2\bar{f}(x_2)}}{\module{x_1-x_2}^{2\alpha}}.
\end{equation*} 

We state the local well-posedness by adapting the arguments in \cite[Theorem 3.1]{Biro2000}
\begin{proposition}\label{p-lwp}
Fix any liquid volume $V$, for problem \eqref{eq-main-syst}  with initial data $(\gdens_0,R_0,\ptR_0)$, where $R_0>0$ and $\gdens_0\in C^{2+2\alpha}\Paren{[0,R_0];(\eta,\infty)}$ for some $\alpha\in(0,1/2)$ and a constant $\eta>0$, there exists a unique solution $R\in C^{3+\alpha}_t([0,\delta])$ and $\gdens\in C^{1+\alpha}_t([0,\delta];C^{2+2\alpha}_r([0,R(t))))$, where $\delta=\delta(\norm{\gdens_0}_{C^{2+2\alpha}})>0$.
\end{proposition}
\begin{proof} 
By changing the variable $x=R(t)y$, we reduce problem \eqref{eq-main-syst} on $B_{\bR(t)}$ to a problem on the fixed domain
\begin{subnumcases}	
{\label{eq-app-2}} 
\parl_t\brho(y,t)=\frac{\Thermal\Laplace_y\log\brho(y,t)}{\Adi \Heatcapa R^2}+\frac{\dot{p}}{\Adi p}\Paren{\frac{y\parl_y\brho(y,t)}{3}+\brho(y,t)},\  \module{y}\le 1,&$t>0$, \label{eq-app-2-1}\\
\ptR=-\frac{\Thermal\parl_y\brho(1,t)}{\Adi\Heatcapa R\brho(1,t)^2}-\frac{R\dot{p}}{3\Adi p},\quad p(t)=\GASCONST\Tconst\brho(1,t),&$t>0$,\\
\brho(1,t)=\frac1{\GASCONST\Tconst}\left[4\mu_l\Paren{\frac{\ptR}{R}-\frac{\ptbR}{\bR}}+\frac{2\sigma}{R}+\frac{2\barsigma}{\bR}\right.\nonumber\\
\qquad\qquad\qquad\ \ \left.+\ldens\Paren{\frac{\bR-R}{\bR}R\pttR+\Paren{\frac32-\frac{2R}{\bR}+\frac{R^4}{2\bR^4}}\ptR^2}\right],&$t>0$,\label{eq-app-2-3} 
\end{subnumcases}
where $\brho(y,t)=\gdens(R(t)y,t)$.
Compared to the proof in \cite[Theorem 3.1]{Biro2000}, the extra terms in \eqref{eq-app-2-3} are analytic in $R>0$ and $\ptR$, since $\bR=\Paren{R^3+\bV}^{\frac 13}>\bV^{\frac 13}$ and $\ptbR =R^2\ptR /\bR^2$. Thus, \eqref{eq-app-2-3} can be written in the same form as equation (3.18) in \cite[Theorem 3.1]{Biro2000}, i.e., an analytic function of $R,\ptR,\pttR$ and $t$. Then, \eqref{eq-main-syst} can be treated as that in \cite[Theorem 3.1]{Biro2000},  and one can follow the same procedure to complete the proof.
\end{proof}  

The global well-posedness and stability results are derived following the proof in \cite{Biro2000} (see also \cite[Section 6]{Lai2023}),  where the global well-posedness of the free boundary problem \eqref{eq-main-syst} is established under conditions where the liquid volume is infinite, there are no viscous terms at the gas-liquid interface, and the external far-field pressure is constant. Additionally, they showed the Lyapunov stability of the problem when the initial data are sufficiently close to a spherically symmetric equilibrium.

For the problem we are considering here for which the liquid has a finite volume $V$, the presence of viscous terms on the gas-liquid interface (see \eqref{eq-b-main-3}) and the external liquid free surface (see \eqref{eq-b-main-5}) lead to the boundary condition \eqref{eq-main-syst-3}. Consequently, a negative term,  $- 16\pi\mu_l \bV (\ptR (t))^2 R(t)/(R(t)^3+\bV)$, appears on the right-hand side of the energy dissipation law \eqref{eq-energy-dissipation-sph}. As a result, the arguments used in the proof of \cite[Theorem 4.1]{Biro2000}, particularly those involving key energy dissipation estimates (4.16) and (4.40) in \cite{Biro2000}, remain applicable.

More precisely, fix the gas mass $M$ and the liquid volume $V$, given $\vare_0>0$, there exists $\eta_0=\eta_0(\vare_0)>0$ such that the following holds: for any mass-preserving ($M=\GASMASS[\gdens_0,R_0]$) initial data $(\gdens_0,R_0,\dot R_0)$ with liquid volume $V$ satisfying 
\begin{equation}\label{eq-2000-1} 
\norm{\gdens_0(R_0y)-\ssd[M,V]}_{C^{2+2\alpha}_y(B_1)}+\module{R_0-\ssr[M,V]}+|\ptR_0|\le\eta_0,
\end{equation}
the global-in-time solution satisfies
\begin{align} 
&\norm{\gdens(R(t)y,t)-\ssd[M,V]}_{C^{2+2\alpha}_y(B_1)}+\module{R(t)-\ssr[M,V]}+|\ptR|\le\vare_0,\label{eq-2000-2}\\
&\norm{p_g}_{C^{1+\alpha}_t(\mathbb{R}_+)}\le\vare_0,\ \text{and}\ \norm{R(t)-\ssr[M,V]}_{C^{3+\alpha}_t(\mathbb{R}_+)}\le\vare_0,\ \forall t>0.\label{eq-2000-3}
\end{align} 
Therefore, given the gas mass $M$ and the liquid volume $V$, we have the spherical equilibrium Lyapunov stability for mass-preserving and volume-invariant perturbations to the equilibrium $(\ssd[M,V],\ssr[M,V])$.
\section{Center manifold theory}\label{appendix-4} 

In this short appendix, we recall the center manifold theory developed in \cite{Lai2023}, and we refer to \cite{Carr1981} for more information on this topic. 
Consider the following equation on a Banach space $\Paren{Z,\norm{\cdot}}$
\begin{equation}\label{eq-app-4-1}
\dot\zz=\Linearpart\zz+\Nonlpart(\zz, \dot\zz),\quad \zz(0)\in Z, 
\end{equation} 
where $\Nonlpart(\zz,\ww):Z\times Z\rightarrow Z$ has a uniformly continuous second order derivative satisfying $\Nonlpart(\Zero,\ww)=\Zero=\parl_{(\zz,\ww)}\Nonlpart(\Zero,\Zero)=\Zero$. We further assume that
\begin{enumerate}[(i)]
\item $Z=X\oplus Y$, where $X$ is a finite-dimensional $\Linearpart$-invariant subspace and $Y$ is a closed  $e^{\Linearpart t}$-invariant subspace. 
\item All the eigenvalues of $\Linearpart|_X$ have zero real parts. 
\item Let $\mathcal{Q}_1:Z\rightarrow X$ be a projection, and $\mathcal{Q}_2=\identity_X-\mathcal{Q}_1$. There exist positive constants $\varpi$ and $c$, such that
\begin{equation}\label{eq-app-4-2}
\norm{e^{\Linearpart t}\mathcal{Q}_2}_{Y\rightarrow Y}\le ce^{-\varpi t},\quad t\ge0.
\end{equation}  
\end{enumerate}

Decompose a solution to \eqref{eq-app-4-1} as $\zz=\xx+\yy=\mathcal{Q}_1\zz+\mathcal{Q}_2\zz$. Then, \eqref{eq-app-4-1} becomes
\begin{equation}\label{eq-app-4-3}
\begin{cases}
\dot\xx=\Linearpart|_X\xx+f(\xx,\yy,\dot\xx,\dot\yy),&\text{where}\ f(\xx,\yy,\dot\xx,\dot\yy)=\mathcal{Q}_1\Nonlpart(\xx+\yy,\dot\xx+\dot\yy),\\
\dot\yy=\mathcal{Q}_2\Linearpart\yy+g(\xx,\yy,\dot\xx,\dot\yy),&\text{where}\ g(\xx,\yy,\dot\xx,\dot\yy)=\mathcal{Q}_2\Nonlpart(\xx+\yy, \dot\xx+\dot\yy).
\end{cases}
\end{equation} 

Recall that an invariant manifold for \eqref{eq-app-4-3} is a curve $\yy = h(\xx)$, defined for $|\xx|$ small, such that the solution $(\xx(t),\yy(t))$ passing through $(\xx(0),h(\xx(0)))$ satisfies $\yy(t)=h(\xx(t))$. A center manifold is an invariant manifold that is tangent to the subspace $X$ at the origin. Let $\mathcal M$ be a center manifold for \eqref{eq-app-4-3} given by $\yy=h(\xx)$. The equation on the center manifold is given by
\begin{equation}\label{eq-app-4-4}
\dot\xx=\Linearpart|_X\xx+f(\xx,h(\xx),\dot\xx,h^\prime(\xx)\dot\xx). 
\end{equation} 
Assume that $\zz(t)$ converges to some point in $\mathcal M$ as $t\rightarrow\infty$, and that $\sup_{t\ge 0}\norm{\dot\zz(t)}$ is sufficiently small. Then, the following lemma holds: 
\begin{lemma}\label{l-e-converge}
Denote by $(\xx(t),\yy(t))$ a solution of \eqref{eq-app-4-3}. Assume that there exists $\vare>0$ such that if $\norm{(\xx(0),\yy(0))}<\vare$, then $\norm{(\dot\xx(t),\dot\yy(t))}<\vare$  for any $t\ge0$.
Then there exist constants $C_1,\varpi_1>0$ such that  $\norm{\yy(t)-h(\xx(t))}\le C_1e^{-\varpi_1t}\norm{\yy(0)-h(\xx(0))}$ for $t\ge0$. 

If we further assume that the zero solution of \eqref{eq-app-4-4} is Lyapunov stable, then there exists a solution $\hat{\xx}(t)$ of \eqref{eq-app-4-4} such that  $\xx(t)=\hat{\xx}(t)+O(e^{-\varpi_0t})$ and $\yy(t)=h(\hat{\xx}(t))+O(e^{-\varpi_0 t})$, as $t\rightarrow\infty$, where $\varpi_0=\min(\varpi,\varpi_1)$ and $\varpi$ is given in \eqref{eq-app-4-2}.
\end{lemma}  
\section{Asymptotic stability in $\ell^2$ and the spectrum analysis}\label{appendix-3}

We first reformulate the asymptotic stability result in the first part of Theorem \ref{t-main result} for system \eqref{eq-l-main}--\eqref{eq-b-main} by using the variable $\zz$ to express it in terms of the equivalent system \eqref{eq-Lw+Nw} within Banach space $\ell^2$.
The proof involves adapting  \cite[Proposition 9.2]{Lai2023}, as it primarily addresses the results related to \eqref{eq-main-syst-1} and \eqref{eq-main-syst-2}. 

\begin{lemma}\label{l-reformulate-L^2}
Fix any gas mass $M$ and liquid volume $V$. For any initial data $(\gdens_0, R_0,\ptR_0)\in C^{2+2\alpha}_r(B_{R_0})\times\mathbb{R}_+\times\mathbb{R}$ with $\alpha\in(0,1/2)$ such that $\GASMASS[\gdens_0,R_0]=M$, denote $\Errorrho_{2}(0)=\gdens_0(R_0)-\ssd[M,V],\ErrorR(0)=R_0-\ssr[M,V]$, and $\theta_k(0)=\int_{B_1}(\gdens_0(R_0y)-\gdens_0(R_0))\Xi_k(y)dy$, where $\Xi_k$ is defined in Proposition \ref{p-infinite-ODE}. Then, we have $\zz(0)=(\Errorrho_{2}(0),\ErrorR(0),\ptR_0,\theta_1(0),\theta_2(0),\cdots)^\top\in\ell^2$, and the sequence $\{j^2\theta_j(0)\}_{j=1}^\infty\in\ell^2$.  

Furthermore, assume that \eqref{eq-main result-condition} holds. Let $(\gdens, R)\in C^{2+2\alpha,1+\alpha}_{r,t}(B_{R(t)}\times[0,\infty))\times C^{3+\alpha}_t$ be the global solution of \eqref{eq-main-syst} with initial data $(\gdens_0, R_0,\ptR_0)$ and liquid volume $V$ as in Theorem \ref{t-main result}. Let $\zz$ be the corresponding solution to system \eqref{eq-Lw+Nw}. Then, it follows that $\{j^2\theta_j\}_{j=1}^\infty\in\ell^2$ and $\norm{\dot\zz}_{\ell^2}+\norm{\zz(t)}_{\ell^2}\rightarrow 0$, as $t\rightarrow\infty$.
\end{lemma}  

Then, we analyze the spectrum of the operator $\Linearpart$ on the space $\ell^2$ using the Laplace transform. Given a function $f(t)$ defined for $t\ge0$, we denote the Laplace transform by $\tilde{f}(\tau) = \int_0^\infty e^{-t\tau}f(t)dt$.
\begin{proof}[Proof of Proposition \ref{p-spectrum}]
Note that the linear system $\dot\zz = \Linearpart\zz$ is equivalent to the linear part of system \eqref{eq-IBVP}
\begin{numcases}
{}\parl_t\Errorrho_{1}=\bkappa\Laplace_y\Errorrho_{1}-\Paren{1-{\Adi}^{-1}}\dot\Errorrho_{2},\quad 0\le y\le1,\quad\Errorrho_{1}(1,t)=0,&$t>0$,\label{eq-app-3-1}\\
\ptErrorR=-\bkappa\ssr\ssd^{-1}\parl_y\Errorrho_{1}(1,t)-\ssr\Adi^{-1}\ssd^{-1}\dot\Errorrho_{2}/3,&$t>0$,\label{eq-app-3-2}\\
\Errorrho_{2}=\Paren{\GASCONST\Tconst}^{-1}(-A\ErrorR+B\dot\ErrorR+\ldens\tR\ddot\ErrorR),&$t>0$,\label{eq-app-3-3}
\end{numcases} 
where 
\begin{equation*}
A=2\sigma/\ssr^2+2\barsigma\ssr^2/\ssor^{4}\ \text{and}\  B=4\mu_l\bV/\ssr(\ssr^3+\bV).
\end{equation*}
Similar to the proof as in Proposition \ref{p-infinite-ODE}, substituting the decomposition \eqref{eq-u-expansion} into \eqref{eq-app-3-1} and testing by $\Xi_k$, it holds $\dot\theta_k(t)=-\bkappa\eigv_k\theta_k(t)-\coeff_k\dot\Errorrho_{2}(t)$.
Taking the Laplace transform yields
\begin{equation}\label{eq-app-3-4}
\widetilde{\theta}_k(\tau)=\frac{\theta_k(0)+\coeff_k\Errorrho_{2}(0)}{\bkappa\eigv_k+\tau}-\frac{\coeff_k\tau}{\bkappa\eigv_k+\tau}\widetilde\Errorrho_{2}(\tau).
\end{equation} 
Again by the Laplace transform, \eqref{eq-app-3-2} and \eqref{eq-app-3-3} become
\begin{numcases}
{}-\ErrorR(0)+\tau\widetilde\ErrorR(\tau)=\sum_{j=1}^\infty\widetilde\theta_j(\tau)\omega_j-3^{-1}\ssr\Adi^{-1}\ssd^{-1}(-\Errorrho_{2}(0)+\tau\widetilde\Errorrho_{2}(\tau)),\label{eq-app-3-5}\\
(\ldens\tR\tau^2+B\tau-A)\widetilde\ErrorR(\tau)-\GASCONST\Tconst\widetilde\Errorrho_{2}(\tau)=\ldens\tR(\ptErrorR(0)+\tau\ErrorR(0)).\label{eq-app-3-6}
\end{numcases} 
Substituting \eqref{eq-app-3-4} and \eqref{eq-app-3-6} into  \eqref{eq-app-3-5} and using  $\coeff_j\omega_j=\frac{2(\Adi-1)\bkappa\ssr}{\Adi\ssd}$, we obtain
\begin{align*}
-\ErrorR(0)+\tau\widetilde\ErrorR(\tau)
={}&-\Paren{\GASCONST\Tconst}^{-1}\Bracket{(\ldens\tR\tau^2+B\tau-A)\tau\widetilde\ErrorR(\tau)-\ldens\tR\tau(\ptErrorR(0)+\tau\ErrorR(0))}\\
&\cdot\Paren{\sum_{k=1}^\infty\frac{\coeff_k\omega_k}{\bkappa\eigv_k+\tau}+\frac{\ssr}{3\Adi\ssd}}+\Bracket{\sum_{k=1}^\infty\frac{\theta_k(0)+\coeff_k\Errorrho_{2}(0)}{\bkappa\eigv_k+\tau}\omega_k+\frac{\ssr}{3\Adi\ssd}\Errorrho_{2}(0)}.
\end{align*}
Therefore, we obtain $\tau\widetilde\ErrorR(\tau)\mathbb{M}(\tau)=\mathbb{S}(\tau)$, where $\mathbb{M}(\tau)$ is defined in \eqref{eq-meromorphic-function} and 
\begin{align*}
\mathbb{S}(\tau)={}&4\pi\frac{\ssd}{\ssr}\left[\frac{\ldens\tR\tau}{\GASCONST\Tconst}\Paren{\sum_{k=1}^\infty\frac{\coeff_k\omega_k}{\bkappa\eigv_k+\tau}+\frac{\ssr}{3\Adi\ssd}}(\ptErrorR(0)+\tau\ErrorR(0))\right.\\
&\qquad\quad+\left.\sum_{k=1}^\infty\frac{\theta_k(0)+\coeff_k\Errorrho_{2}(0)}{\bkappa\eigv_k+\tau}\omega_k+\frac{\ssr}{3\Adi\ssd}\Errorrho_{2}(0)+\ErrorR(0)\right] 
\end{align*}
is analytic for all $\tau\in\mathbb{C}\setminus\{-\pi^2\bkappa j^2:j=1,2,\cdots\}$ since $-\bkappa\eigv_j=-\pi^2\bkappa j^2$. Then, \eqref{eq-app-3-4} and \eqref{eq-app-3-6} yield that $\widetilde\zz=(\widetilde\Errorrho_{2}(\tau),\widetilde\ErrorR(\tau),\widetilde{\ptErrorR}(\tau),\widetilde{\theta}_1(\tau),\widetilde{\theta}_2(\tau),\cdots)^\top=(\Linearpart-\tau I)^{-1}\widetilde\zz(0)$ for $\tau\in\mathbb{C}\setminus\{ -\pi^2\bkappa j^2:j=1,2,\cdots\}$ satisfying $\tau \mathbb{M}(\tau)\neq0$. 

To estimate the upper bound of $\sup\{\operatorname{Re}(\lambda):\lambda\in\operatorname{sp}(\Linearpart)\setminus\{0\}\}$, we rewrite 
\begin{equation*}
\mathbb{M}(\lambda)=4\pi\frac{\ssd}{\ssr}+\frac{1}{\GASCONST\Tconst}\Paren{\sum_{k=1}^\infty\frac{1}{\Adi}\frac{8(\Adi-1)\pi\bkappa}{\bkappa\pi^2 k^2+\lambda}+\frac{4\pi}{3\Adi}}\Paren{C\lambda^2+B\lambda -A},
\end{equation*}
where $C=\ldens(\ssr-{\ssr^2}/{\ssor})$. The remaining proof of the lemma is a consequence of Lemma \ref{l-real-part} below. 
\end{proof} 
\begin{lemma}\label{l-real-part}
There exists a constant $\varpi>0$ such that $x<-2\varpi$ for all complex roots $\lambda=x+iy$ to the function $\mathbb{M}$ in \eqref{eq-meromorphic-function}.
The constant $\varpi$ can be chosen as 
\begin{subnumcases}
{\varpi=}
\frac{1}{2}\min\Brace{\Theta_2\pi^2\bkappa,\max\Brace{\frac{B}{2C},\min\Brace{\Theta_1\pi^2\bkappa,\sqrt{\frac{K}{2}}}}},&$B^2\le 4KC^2$,\nonumber\\
\frac{1}{2}\min\Brace{\Theta_2\pi^2\bkappa,\frac{B-\sqrt{B^2-4KC^2}}{2C}},&$B^2>4KC^2$,\nonumber
\end{subnumcases} 
where the constants $\Theta_1$ and $\Theta_2$ are defined in \eqref{eq-Theta-1} and \eqref{eq-Theta-2}, respectively. Moreover, $K=2\GASCONST\Tconst\ssd\ssor/[\ldens\ssr^2(\ssor-\ssr)]$.
\end{lemma}
\begin{proof} 
Let $\lambda=x+iy$ be a complex root of $\mathbb{M}$. Substituting $\lambda=x+iy$ into \eqref{eq-meromorphic-function} and splitting the real and imaginary parts of $\mathbb{M}$, we obtain
\begin{subnumcases}
{}\frac1{\GASCONST\Tconst}\left\lbrace\Bracket{\frac{4\pi}{3\Adi}+\sum_{j=1}^\infty\frac{1}{\Adi}\frac{8(\Adi-1)\pi^3\bkappa^2j^2}{\Paren{\pi^2\bkappa j^2+x}^2+y^2}}\Bracket{C\Paren{x^2-y^2}+Bx-A}\right.\nonumber\\
\qquad\ \ \ \left.+\sum_{j=1}^\infty\frac{1}{\Adi}\frac{8(\Adi-1)\pi\bkappa}{\Paren{\pi^2\bkappa j^2+x}^2+y^2}\Bracket{(Cx+B)(x^2+y^2)-Ax}\right\rbrace+4\pi\frac{\ssd}{\ssr}=0,\label{eq-Re-Q}\\
\frac{i}{\GASCONST\Tconst}\left\lbrace\Bracket{\frac{4\pi}{3\Adi}+\sum_{j=1}^\infty\frac{1}{\Adi}\frac{8(\Adi-1)\pi\bkappa\Paren{\pi^2\bkappa j^2+x}}{\Paren{\pi^2\bkappa j^2+x}^2+y^2}}\Paren{2Cx+B}y\right.\nonumber\\
\qquad\ \ \ \left.-\sum_{j=1}^\infty\frac{1}{\Adi}\frac{8(\Adi-1)\pi\bkappa}{\Paren{\pi^2\bkappa j^2+x}^2+y^2}\Bracket{C\Paren{x^2-y^2}+Bx-A}y\right\rbrace=0.\label{eq-Im-Q}
\end{subnumcases} 
To simplify the above formulas, we will denote
$D=\frac{4\pi}{3\Adi}+\sum_{j=1}^\infty\frac{1}{\Adi}\frac{8(\Adi-1)\pi\bkappa\Paren{\pi^2\bkappa j^2+x}}{\Paren{\pi^2\bkappa j^2+x}^2+y^2},E=\sum_{j=1}^\infty\frac{1}{\Adi}\frac{8(\Adi-1)\pi^3\bkappa^2j^2}{\Paren{\pi^2\bkappa j^2+x}^2+y^2}$, and $F=\sum_{j=1}^\infty\frac{1}{\Adi}\frac{8(\Adi-1)\pi\bkappa}{\Paren{\pi^2\bkappa j^2+x}^2+y^2}$. 

\noindent \textbf{Case 1: $0<y^2\le K/2$.}  If $y\neq0$,  eliminating $y$ for the imaginary part \eqref{eq-Im-Q} yields
\begin{equation}\label{eq-Im-Q-1}
0=\Paren{\frac{4\pi}{3\Adi}+E}\Paren{2Cx+B}+F\Bracket{C(x^2+y^2)+A}.
\end{equation}  
From $E>0$ and $F\Bracket{C(x^2+y^2)+A}>0$, it holds 
\begin{equation}\label{eq-Im-Q-2}
x<-\frac{B}{2C}<0, 
\end{equation} 
provided $y\neq 0$. Splitting $F$ in \eqref{eq-Im-Q-1} and substituting into the real part  \eqref{eq-Re-Q}, we derive
\begin{align}
&\Paren{\frac{4\pi}{3\Adi}+E}\Brace{C(x^2-y^2)+Bx-A-\frac{\Paren{2Cx+B}\Bracket{(Cx+B)(x^2+y^2)-Ax}}{C(x^2+y^2)+A}}\nonumber\\
={}&-4\pi\GASCONST\Tconst\frac{\ssd}{\ssr}.\label{eq-Re-Q-1}
\end{align}  
Denote the term in the big brace by $\omega$, and a straightforward calculation shows that 
\begin{align*}
&\Bracket{C(x^2+y^2)+A}\omega\\
={}&-C^2(x^2+y^2)^2-B\Paren{2Cx+B}(x^2+y^2)+A\Paren{2Cx^2+2Bx-2Cy^2}-A^2\\
>{}&-C^2(x^2+y^2)^2-2AC(x^2+y^2)-A^2=-\Bracket{C(x^2+y^2)+A}^2.
\end{align*}  
In the above, we have used \eqref{eq-Im-Q-2} to deduce $2Cx^2+2Bx=x\Paren{2Cx+2B}>-x(2Cx)$. Then, one has $\omega>-C(x^2+y^2)-A$, and \eqref{eq-Re-Q-1} yields
\begin{equation}\label{eq-Re-Q-2}
\frac1{\GASCONST\Tconst}\Paren{\frac{4\pi}{3\Adi}+E}\Bracket{C(x^2+y^2)+A}> 4\pi\frac{\ssd}{\ssr}.
\end{equation} 

Assume that $x\ge-\Theta\pi^2\bkappa$, where $\Theta\in(0,1)$ will be chosen later. It is clear that $x\ge-\Theta\pi^2\bkappa j^2$ for all $j\ge 1$. By equilibrium equation \eqref{eq-equi-eq} and recalling $\ssr<\ssor$, it holds
\begin{equation}\label{eq-KC+A}
\frac{2\GASCONST\Tconst\ssd}{\ssr}<KC+A=\frac{2\GASCONST\Tconst\ssd}{\ssr}+\frac{2\sigma}{\ssr^2}+\frac{2\barsigma\ssr^2}{\ssor^{4}}<\frac{3\GASCONST\Tconst\ssd}{\ssr}.
\end{equation}
We further assume that 
\begin{equation}\label{eq-ass-1}
x\ge-\sqrt{K-y^2},
\end{equation} 
if $0<y^2\le K$. Since \eqref{eq-Im-Q-2} implies $x<0$, we have $C(x^2+y^2)+A\le KC+A$, if $0<y^2\le K$. Then \eqref{eq-Re-Q-2} yields 
\begin{equation*}
0>-\frac1{\GASCONST\Tconst}\Paren{\frac{4\pi}{3\Adi}+\frac{8(\Adi-1)}{\pi\Adi}\frac1{(1-\Theta)^2}\sum_{j=1}^\infty j^{-2}}\Paren{KC+A}+4\pi\frac{\ssd}{\ssr},
\end{equation*}
where we have substituted $x\ge-\Theta\pi^2\bkappa j^2$.  Upon simplification the above inequality, one has 
\begin{equation}\label{eq-Theta-1}
\Theta>\Theta_1\triangleq 1-\sqrt{\Paren{\Adi-1}\Big/\Paren{\frac{3\GASCONST\Tconst\ssd\Adi}{\ssr(KC+A)}-1}}\in(0,1).
\end{equation}

We simply choose $\Theta =\Theta_1$ to reach a contradiction to \eqref{eq-ass-1}, since we will deduce $\Theta>\Theta_1$. Thus, we have for $0<y^2\le K$ that $x<-\min\Brace{\Theta_1\pi^2\bkappa,\sqrt{K-y^2}}$. Combining \eqref{eq-Im-Q-2}, we have
\begin{equation}\label{eq-ass-2}
x<-\max\Brace{\frac{B}{2C},\min\Brace{\Theta_1\pi^2\bkappa,\sqrt{\frac{K}{2}}}},\ \text{if}\ 0<y^2\le \frac{K}{2}.
\end{equation}
\noindent \textbf{Case 2: $y^2> K/2$.} We eliminate $D$ in the imaginary part \eqref{eq-Im-Q} and substitute it into the real part \eqref{eq-Re-Q}
\begin{equation}\label{eq-ass-3}
4\pi\ssd\ssr^{-1}\GASCONST\Tconst\Paren{2Cx+B}=-F\module{C\lambda^2+B\lambda-A}^2.
\end{equation} 
Note that $A,B,C>0$ and therefore $\module{C\lambda^2+B\lambda-A}^2=C^2\module{\lambda-\lambda_1}^2\module{\lambda-\lambda_2}^2$ has real roots $\lambda_{i}\in\mathbb{R},i=1,2$. Then, we have $\module{C\lambda^2+B\lambda-A}^2>C^2y^4$ since $|\lambda-\lambda_{i}|=|x+iy-\lambda_{i}|>|iy|$. Using $y^2>K/2$, we further derive
\begin{equation}\label{eq-ass-4}
4\pi\ssd\ssr^{-1}\GASCONST\Tconst\Paren{2Cx+B}<-C^2\Paren{\frac{K}{2}}^2\sum_{j=1}^\infty\frac{1}{\Adi}\frac{8(\Adi-1)\pi\bkappa}{\pi^4\bkappa^2j^4+K/2}. 
\end{equation}  
Then, recalling the definitions of $C,K$ and $\bkappa$, we have from \eqref{eq-ass-4} that  
\begin{equation}\label{eq-ass-5}
x
<-\frac{B}{2C}-\frac{(\Adi-1)\GASCONST\Tconst}{\pi^4\ldens\Thermal}\frac{\ssd^2}{1-\ssr/\ssor}\Paren{\sum_{j=1}^\infty\frac{1}{j^4+\frac{K}{2\pi^4\bkappa^2}}}<-\frac{B}{2C},\ \text{if}\ y^2>\frac{K}{2}.
\end{equation}  
\noindent \textbf{Case 3: $y=0$.} 
We claim that $x<0$. Suppose $x\ge 0$, and we rewrite $\mathbb{M}(x)$ as
\begin{align*}
\frac{\pi x(Cx+B)}{\GASCONST\Tconst\Adi}\Paren{\frac{4}{3}+\sum_{j=1}^\infty\frac{8(\Adi-1)\bkappa}{\pi^2\bkappa j^2+x}}+\Bracket{\frac{4\pi\ssd}{\ssr}-\frac{\pi A}{\GASCONST\Tconst\Adi}\Paren{\frac{4}{3}+\sum_{j=1}^\infty \frac{8(\Adi-1)\bkappa}{\pi^2\bkappa j^2+x}}}.
\end{align*} 
From $x\ge 0$, the first term is non-negative, and we apply \eqref{eq-equi-eq} to deduce that the term in the bracket is greater than
\begin{align*}
&\frac{8\pi}{\GASCONST\Tconst}\Paren{\frac{\sigma}{\ssr^2}+\frac{\barsigma}{\ssr\ssor}}-\frac{\pi}{\GASCONST\Tconst\Adi}\Paren{\frac{2\sigma}{\ssr^2}+\frac{2\barsigma\ssr^2}{\ssor^{4}}}\Paren{\frac{4}{3}+\frac{8(\Adi-1)\bkappa}{\pi^2\bkappa}\sum_{j=1}^{\infty}j^{-2}}\\
\ge{}&\frac{8\pi}{\GASCONST\Tconst}\Paren{\frac{\sigma}{\ssr^2}+\frac{\barsigma}{\ssr\ssor}}-\frac{1}{3}\frac{8\pi}{\GASCONST\Tconst}\Paren{\frac{\sigma}{\ssr^2}+\frac{\barsigma\ssr^2}{\ssor^{4}}}>0.
\end{align*}  
This contradicts $\mathbb{M}(x)=0$. To search for a negative upper bound, we assume that $x \ge -\Theta^\prime\pi^2\bkappa$, 
where $0<\Theta^\prime<1$ will be chosen. We further assume that 
\begin{equation}\label{eq-ass-6}
x>\frac{-B+\sqrt{B^2-4KC^2}}{2C},
\end{equation}
if $B^2>4KC^2$. In this case, $Cx^2+Bx-A\ge -KC-A$, and this inequality also holds when $B^2\le 4KC^2$.
Then, we have 
\begin{align*}
0=\mathbb{M}(x)
&>-\frac1{\GASCONST\Tconst}\Paren{\frac{4\pi}{3\Adi}+\frac{8(\Adi-1)}{\pi\Adi(1-\Theta^\prime)}\sum_{j=1}^\infty j^{-2}}\Paren{KC+A}+4\pi\frac{\ssd}{\ssr}\\
&=-\frac1{\GASCONST\Tconst}\Paren{\frac{4\pi}{3\Adi}+\frac{4\pi}3\Paren{1-\frac1\Adi}\frac1{1-\Theta^\prime}}\Paren{KC+A}+4\pi\frac{\ssd}{\ssr}.
\end{align*} 
Recalling \eqref{eq-KC+A}, it follows that $3\GASCONST\Tconst{\ssd}/{\ssr}<[\Adi^{-1}+\Paren{1-\Adi^{-1}}\Paren{1-\Theta^\prime}^{-1}]\Paren{KC+A}$, or equivalently,
\begin{equation}\label{eq-Theta-2}
\Theta^\prime>\Theta_2\triangleq1-\Paren{\Adi-1}\Big/\Paren{\frac{3\GASCONST\Tconst\Adi\ssd}{\ssr(KC+A)}-1}\in(0,1).
\end{equation}
Therefore, we reach a contradiction to \eqref{eq-ass-6} by choosing $\Theta^\prime=\Theta_2$. 

Summing up, in the case of $y=0$, $x<-\min\Brace{\Theta_2\pi^2\bkappa,\frac{B-\sqrt{B^2-4KC^2}}{2C}}$, if $B^2>4KC^2$. Otherwise, $x<-\Theta_2\pi^2\bkappa$. This, combined with the upper bounds \eqref{eq-ass-2}, \eqref{eq-ass-5} and the fact $\Theta_2>\Theta_1$ gives the upper bound.
\end{proof}

\section*{Acknowledgments}
C. Hao and S. Yang were partially supported by the NSF of China under Grant 12171460; C. Hao was additionally supported by the CAS Project for Young Scientists in Basic Research under Grant YSBR-031 and the National Key R\&D Program of China under Grant 2021YFA1000800. T. Luo was supported by a General Research Fund of Research Grants Council of the Hong Kong (Grant No. 11310023).

\end{document}